\documentclass[letterpaper,11pt]{amsart}

\usepackage{amssymb}
\usepackage{amsmath}
\usepackage{amsthm}
\usepackage{enumerate}
\usepackage{graphicx}
\usepackage{blkarray}

\usepackage{caption}
\usepackage{subcaption}

\usepackage[margin=1.3in]{geometry}
\usepackage{hyperref}  
 
\hypersetup{colorlinks=true, citecolor=blue}
\theoremstyle{plain}
\newtheorem{thm}{Theorem}[section]
\newtheorem{prop}[thm]{Proposition}
\newtheorem{lem}[thm]{Lemma}
\newtheorem{cor}[thm]{Corollary}

\numberwithin{equation}{section}

\theoremstyle{definition}

\theoremstyle{remark}
\newtheorem{remark}[thm]{Remark}

\newtheorem*{acknowledgements}{Acknowledgements}

\theoremstyle{plain}
\newtheorem{introthm}{Theorem}

\newcommand{\thmref}[1]{Theorem~\ref{#1}}
\newcommand{\propref}[1]{Proposition~\ref{#1}}
\newcommand{\secref}[1]{Section~\ref{#1}}
\newcommand{\subsecref}[1]{subsection~\ref{#1}}
\newcommand{\lemref}[1]{Lemma~\ref{#1}}
\newcommand{\corref}[1]{Corollary~\ref{#1}}
\newcommand{\figref}[1]{Figure~\ref{#1}}

\newcommand{\remref}[1]{Remark~\ref{#1}}
\newcommand{\eqnref}[1]{Equation~\eqref{#1}}

\newcommand{\calC}{{\mathcal C}}

\newcommand{\M}{{\mathcal M}}

\newcommand{\calP}{{\mathcal P}}

\newcommand{\T}{{\mathcal T}}

\newcommand{\R}{{\mathbb R}}
\newcommand{\Z}{{\mathbb Z}}

\newcommand{\wtilde}{\widetilde}

\newcommand{\st}{{\, \big| \,}}

\newcommand{\param}{{\mathchoice{\mkern1mu\mbox{\raise2.2pt\hbox{$
\centerdot$}}
\mkern1mu}{\mkern1mu\mbox{\raise2.2pt\hbox{$\centerdot$}}\mkern1mu}{
\mkern1.5mu\centerdot\mkern1.5mu}{\mkern1.5mu\centerdot\mkern1.5mu}}}

\DeclareMathOperator{\PSL}{PSL}

\DeclareMathOperator{\sys}{sys}
\DeclareMathOperator{\arcsinh}{arcsinh}
\DeclareMathOperator{\arccosh}{arccosh}

\DeclareMathOperator{\spann}{span}
\DeclareMathOperator{\tree}{\Sigma}
\DeclareMathOperator{\aut}{Aut}
\DeclareMathOperator{\ind}{ind}

\newcommand{\girth}{{w}}

\newcommand{\eps}{{\varepsilon}}
   
\begin{document}

\title{Local maxima of the systole function}

\author{Maxime Fortier Bourque}
\address{School of Mathematics and Statistics, University of Glasgow, University Place, Glasgow, United Kingdom, G12 8QQ}
\email{maxime.fortier-bourque@glasgow.ac.uk}

\author{Kasra Rafi}
\address{Department of Mathematics, University of Toronto, 40 St. George Street, Toronto, ON, Canada M5S 2E4}
\email{rafi@math.toronto.edu}

\begin{abstract}
We construct infinite families of closed hyperbolic surfaces that are local maxima for the systole function on their respective moduli spaces. The systole takes values along a linearly divergent sequence $(L_n)_{n\geq 1}$ at these local maxima. The only surface corresponding to $L_1\approx 3.057$ is the Bolza surface in genus $2$. For every genus $g\geq 13$, we obtain either one or two local maxima in $\mathcal{M}_g$ whose systoles have length $L_2\approx 5.909$. For each $n\geq 3$, there is an arithmetic sequence of genera $(g_k)_{k\geq 1}$ such that the number of local maxima of the systole function in $\mathcal{M}_{g_k}$ at height $L_n$ grows super-exponentially in $g_k$. In particular, level sets of the systole function can have an arbitrarily large number of connected components. Many of the surfaces we construct have trivial automorphism group, and are the first examples of local maxima with this property.
\end{abstract}


\maketitle

\section{Introduction}

The systole of a hyperbolic surface is the length of any of its shortest closed geodesics. For any $g\geq 2$, this defines a continuous function $\sys : \T_g \to \R_+$ on the Teichm\"uller space of closed hyperbolic surfaces of genus $g$ which is invariant under the action of the mapping class group, hence descends to a continuous function on the moduli space $\M_g$. 

By Mumford's compactness criterion \cite{Mumford}, the thick part $\{ x \in \M_g \st \sys(x) \geq \eps \}$ of moduli space is compact for any $\eps > 0$. Therefore, the systole function attains a global maximum on each moduli space. The precise value of the maximum is unknown in general; the best bounds known to date are
\begin{equation} \label{eq:globalmax}
\frac{4}{3} \leq \limsup_{g \to \infty} \frac{\max \{\sys(x) \st x \in \M_g  \}}{\log g} \leq 2.
\end{equation}
The upper bound is a standard area argument, while the lower bound is a result due to Buser and Sarnak \cite{BuserSarnak}.

In genus $2$, the maximum of the systole function is attained at the Bolza surface \cite{Jenni}, which is also the surface of genus $2$ with the largest automorphism group. In fact, the Bolza surface is the only local maximum of $\sys$ in $\M_2$ \cite{SchmutzMaxima}. Surprisingly, Klein's quartic---the surface of genus $3$ with largest automorphism group---is not the global maximizer of $\sys$ in $\M_3$ although it is a local maximum. Schmutz Schaller \cite{SchmutzMaxima} found two other local maxima in $\M_3$---one with larger systole than Klein's quartic---and conjectured that these are the only ones.

In the same paper, Schmutz Schaller constructed examples\footnote{Technically, the proof that these are local maxima contains a gap. It ultimately relies on \cite{SchmutzGerman} which proves that a certain set $F$ of curves is such the map $x \to (\ell_\alpha(x))_{\alpha \in F}$ that records the length of their geodesic representatives is injective on Teichm\"uller space. What is really needed is that the differential of that map is injective at the given surface, which is not shown. Injective smooth maps are not always immersions. Perhaps this cannot happen in the context at hand, but we could not see why.} of local maxima in genus $4$, $5$, $11$, $23$ and $59$, and one additional example in every odd genus. In this infinite family, the systole is bounded above by $5.634$. More sporadic examples in genus $4$, $5$ and $13$ are presented in \cite{SchmutzMoreExamples}. Three further examples of locally maximal triangle surfaces of genus $6$, $10$ and $10$ were discovered in \cite{Ham} and \cite{HamKoch}. Besides the Bolza surface, the only known global maximizers of the systole function are the quotients of the upper half-plane by the principal congruence subgroups of $\PSL(2,\Z)$, which are punctured surfaces \cite{SchmutzCongruence}. All of the above examples have large isometry groups. For instance, they are all regular orbifold covers of hyperbolic polygons.

In this paper, we construct infinite families of local maxima of the systole function. The surfaces are arranged by levels, and to each level (except the first) correspond infinitely many surfaces. The result can be summarized as follows.

\begin{introthm} \label{thmA}
For every integer $n \geq 1$, there exist $L_n>0$ and $w_n>0$ such that for every finite, connected, $n$-regular, signed graph $\Gamma$ of girth at least $w_n$ there exists a closed hyperbolic surface $X(\Gamma)$ which is a local maximum of $\sys$ at height $\sys(X(\Gamma))=L_n$.
\end{introthm}

Here a \emph{signed graph} is a graph equipped with a cyclic ordering of the edges adjacent to any vertex and a sign attached to any two consecutive edges in the cyclic order such that the product of the signs around any vertex is negative. See \subsecref{subsec:signed_graph} for more details.

The genus of the surface $X(\Gamma)$ is equal to $E + 1$ where $E$ is the number of edges in $\Gamma$ (\thmref{thm:systoles_closed_surface}), and the map $\Gamma \mapsto X(\Gamma)$ is injective on the set of isomorphism classes of signed graphs (\thmref{thm:auto}). The sequence $(L_n)_{n\geq 1}$ grows asymptotically linearly in $n$ (\lemref{lem:estimate1}), yielding the first explicit examples of local maxima at arbitrarily large heights.

For $n=1$, there is only one connected $1$-regular signed graph $\Gamma$ and the corresponding surface $X(\Gamma)$ is the Bolza surface (\thmref{thm:bolza}).

{For $n=2$, we obtain two local maxima in every genus $g \geq 13$ which is congruent to $1 \mod 3$ and one local maximum in all other genera larger than $13$ (see  \remref{rem:n=2} and \subsecref{subsec:length}).

For $n \geq 3$, the situation changes drastically. Indeed, the number of connected $n$-regular signed graphs with $(g-1)$ edges grows rapidly with $g$, and an asymptotically positive proportion of them have girth larger than $w_n$. Furthermore, most of them have trivial automorphism group, so that the same holds for the resulting surfaces.

\begin{introthm} \label{thmB}
Let $n\geq3$ and let $g$ be a positive integer such that $2(g-1)/n$ is also an integer. If $g$ is large enough, then the number of local maxima of the systole function in $\M_g$
at height $L_n$ whose automorphism group is trivial is at least
\[
\alpha_n \big(\beta \, g \big)^{\left(1-\frac2n \right) g} 
\] 
where $\beta>0$ is independent of $n$ and $g$, and $\alpha_n>0$ depends only on $n$.
\end{introthm}

In other words, the number of asymmetric local maxima of $\sys$ in $\M_g$ at each fixed height $L_n$ grows super-exponentially along an arithmetic sequence of genera $g$. In particular, level sets of the systole function can have an arbitrarily large number of point components. 

The existence of asymmetric local maxima was conjectured in \cite[p.565]{SchmutzMaxima} and this intuition was repeated in \cite[p.437]{SchmutzMoreExamples}.

\subsection*{Why study the systole function?}

Akrout \cite{Akrout} proved that $\sys$ is a topological\footnote{Note that $\sys$ is not smooth wherever there is more than one systole.} Morse func\-tion on $\T_g$. This implies that in theory one could compute topological invariants of $\M_g$ by finding the Morse singularities of $\sys$ and their indices. For example, the orbifold Euler characteristic of $\M_g$ is given by the formula
\begin{equation} \label{eq:Euler_char}
\chi(\M_g) = \sum_{x \in \calC} \frac{(-1)^{\ind(x)}}{|\aut(x)|}
\end{equation}
where $\calC$ is a set of representatives of the critical points of $\sys$ in $\T_g$ under the action of the mapping class group, $\ind(x)$ is the Morse index of $\sys$ at $x$, and $\aut(x)$ is the group of automorphisms of $x$ \cite{Akrout,SchmutzMorse}. 

Given how difficult it is to identify critical points of $\sys$, \eqnref{eq:Euler_char} is unlikely to be of any practical use. Luckily, the orbifold Euler characteristic of $\M_g$ was computed by Harer and Zagier \cite{HarerZagier} to be
\[
\chi(\M_g) = \frac{B_{2g}}{4g(g-1)}
\]
using different means, where $B_{2g}$ is the $2g$-th Bernoulli number (see also \cite{Penner}). This implies that $|\chi(\M_g)|$ grows roughly like $g^{2g}$, and the number of critical points should be at least as large. That many of these critical points are local maxima is somewhat surprising, and we expect that there are many more than the ones found here. This might indicate that the systole function is inefficient in the sense that it has more singularities than is really required by the topology.

Via its level sets, the systole function gives a ``shape'' to moduli space that interacts with the Teichm\"uller metric, the Weil--Petersson metric and the Thurston metric. Many results in the theory of Riemann surfaces can be stated in terms of the systole function. The first instance is perhaps Mumford's compactness criterion stated earlier. As another example, Keen's collar lemma \cite[p.380]{Primer} and Harvey's observation that the curve graph is connected \cite{Harvey} together imply that the sublevel sets $\{ x \in \M_g \st \sys(x) \leq \eps \}$ are connected if $\eps$ is sufficiently small. That is, moduli space has one end.

More recently, the following results were obtained:
\begin{itemize}
\item the asymptotic cone of $\M_g$ equipped with the Teichm\"uller metric is isometric to the infinite cone over the quotient of the curve complex by the mapping class group \cite{FarbMasur};
\item estimates for the diameter of the thick part of moduli space were given in \cite{CavendishParlier} and \cite{RafiTao};
\item the closed geodesics of length at most $L/g$ in $\M_g$  lie in $\sys^{-1}(I)$ for a fixed compact interval $I \subset \R_+$ \cite{LeiningerMargalit} and the associated peusdo-Anosov homeomorphisms all arise as monodromies of Dehn fillings of finitely many fibered hyperbolic $3$-manifolds \cite{FarbLeiningerMargalit};
\item finite covers of a fixed surface in $\M_2$ are not asymptotically dense in the thick part of $\M_g$ with respect to the Teichm\"uller metric \cite{FletcherKahnMarkovic} nor are they coarsely dense with respect to the Weil--Petersson metric \cite[Theorem 7.1.3]{Doria} (whether they are coarsely dense in the Teichm\"uller metric is an open question of Mirzakhani);
\item the expected value of the systole of a surface in $\M_g$ with respect to the Weil--Petersson volume converges to $1.18915\ldots$ as $g \to \infty$ \cite{MirzakhaniPetri}.
\end{itemize}

Another question of Mirzakhani is whether moduli space has ``long fingers''. We may define the \emph{finger} associated with a local maximum $x$ to be the component $F$ of the super\-level set $\{ y \in \M_g \st \sys(y)>L\}$ containing $x$, where $L$ is the smallest positive number such that $F$ does not contain any other singularities than $x$. The \emph{length} of the finger $F$ is then $\sys(x)-L$. In other words, how large can the total variation of the systole function be between a local maximum and the nearest (with respect to variation) singularity? We do not answer this question here, but the examples in \thmref{thmA} provide a place to start. There are many other open questions related to the systole function; see Section 4 of \cite{Parlier} for instance.

\subsection*{Why study local maxima?}

There are a number of interesting necessary conditions for a closed surface $x$ to be a local maximum of the systole function, namely,
\begin{itemize}
\item every closed geodesic on $x$ must intersect at least two systoles (in particular, the systoles must fill) \cite[Lemma 2.3]{SchmutzMaxima};
\item the systoles must be non-separating \cite[Proposition 2.6]{SchmutzMaxima};
\item there must be at least $(6g-5)$ systoles \cite[Theorem 2.8]{SchmutzMaxima}.
\end{itemize}
Moreover, two systoles on a closed surface can intersect at most once (this is always true, not just at local maxima). Even constructing sets of curves that satisfy all of these topological conditions is a non-trivial task (see \cite{Anderson} for related results).

In an unpublished manuscript, Thurston suggested that the set of surfaces whose systoles fill should form a spine for moduli space, but his proof was incomplete. See \cite{Ji} for a discussion of its limitations and an alternative construction of spines with positive codimension. It seems that little is known about the set of surfaces whose systoles fill. Local maxima of the systole function are contained in this set.

Another naive reason to study local maxima is that the global maximum is among them. One might hope to end up on or near the top of moduli space  by searching for local maxima and perhaps increase the lower bound in \eqref{eq:globalmax}. Unfortunately, the height of our local maxima grows at most like $\log \log g$ rather than $\log g$. It is possible that by using a similar trick as in \cite{Petri}, one could obtain a sequence of examples with systole growing logarithmically in the genus, although we did not explore that idea further.

\subsection*{Proof outline}

As explained in \cite{SchmutzMaxima}, showing that a surface $x$ is a local maximum of the systole function consists in three steps:
\begin{enumerate}
\item finding the set $S$ of systoles of $x$;
\item showing that differential of the vector of lengths of the curves in $S$---a function on Teichm\"uller space---is injective at $x$;
\item proving that under any non-trivial infinitesimal deformation of $x$, at least one of the curves in $S$ shrinks, i.e., has negative differential in that direction. 
\end{enumerate}

Step (2) is necessary \cite[Theorem 2.7]{SchmutzMaxima} and is a problem of general interest \cite[p.378]{Gardiner}. It is equivalent to showing that certain quadratic differentials associated with the curves span the cotangent space to Teichm\"uller space over $\R$. Examples of such bases were described in \cite[Theorem 3.4]{WolpertDuality}. The main tool we use here is the famous cosine formula for the variation of length along twist deformations \cite{WolpertTwist,Kerckhoff}.  The question of which finite sets $F$ of curves are such that their lengths define a global embedding of Teichm\"uller space into $\R_+^F$ is closely related and classical \cite{SchmutzGerman,HamParam1,HamParam2}\cite[p.287]{Primer}. 

The proof of step (3) is easier, but still novel. Whereas Schmutz Schaller relied heavily on symmetries to prove this step, we manage with only virtual symmetries. That is, even though the surfaces we construct have trivial automorphism group in general, they have infinite covers to which all the systoles lift and where the systoles fall into only three orbits under the automorphism group. These covers are the trees of rings of \subsecref{subsec:tree}.

The proofs of (1), (2) and (3) are unified in the sense that they treat all $n\geq1$ and all regular signed graphs $\Gamma$ at once. We believe that the arguments could be applied to similar constructions with different building blocks (the rings of \subsecref{subsec:the_ring}). This is in contrast with \cite{SchmutzMaxima} where a case by case analysis was needed, with some arguments as ad hoc as ``We indicate for all combinatorial possibilities a coefficient of $\zeta$ which is negative.''

\subsection*{Organization}

The paper is organized as follows. \secref{sec:systoles} is devoted to the construction of the surfaces $X(\Gamma)$ and fin\-ding their systoles. Steps (2) and (3) of the above program are carried out in Sections \ref{sec:regularity} and \ref{sec:minimality} respectively, thereby proving \thmref{thmA}. In \secref{sec:isometries}, we show that any orientation-preserving isometry $X(\Gamma_1) \to X(\Gamma_2)$ is induced by a signed graph isomorphism $\Gamma_1 \to \Gamma_2$. Finally, in \secref{sec:counting} we estimate the number of asymmetric surfaces $X(\Gamma)$ constructed in each genus, proving \thmref{thmB}.

\begin{acknowledgements}
We thank Robert Young for suggesting the use of the Gershgorin cir\-cle theorem and Dmitri Gekhtman for pointing out Wolpert's length-twist duality, which together lead to the proof of step (2) (\thmref{thm:lengths_determine}). We also thank Curt McMullen for comments on an earlier draft. MFB and KR were partially supported by Discovery Grants from the Natural Sciences and Engineering Research Council of Canada (RGPIN 06768 and 06486 respectively).  
\end{acknowledgements}

\section{The construction} \label{sec:systoles}

In this section, we construct a highly symmetric surface $R(n,t)$ of genus $1$ with $2n$ boundary components of length $4t$ each, for any $n \geq 1$ and any $t>0$. We then fix a specific value of $t$ for each $n$ and build closed surfaces out of pieces isometric to $R(n,t_n)$. 

\subsection{Trigonometry}

We first gather some trigonometric formulas here for use throughout the paper. See e.g. \cite[p.454]{Buser}.

\vspace{\baselineskip}
\noindent\textbf{Right triangles}

\noindent\begin{minipage}{.5\textwidth}
\begin{align}
 \cosh c &= \cosh a \cosh b \label{eq:righttriangle}\\
 \cos \beta &= \cosh b \sin \alpha \label{eq:righttriangle2}
\end{align}
\end{minipage}
\begin{minipage}{.45\textwidth}
\vspace{2em}
\centering
{\includegraphics{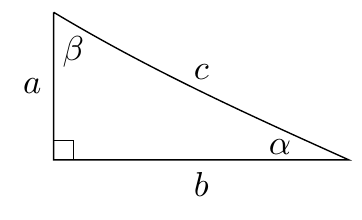}}
\label{fig:righttriangle}
\end{minipage}

\noindent\textbf{Right-angled pentagons}

\noindent\begin{minipage}{.5\textwidth}
\begin{align}
   \cosh c &= \sinh a \sinh b \label{eq:pentsinh} \\ 
   \cosh c &= \coth \alpha \coth \beta \label{eq:pentcoth}
\end{align}
\end{minipage}
\begin{minipage}{.45\textwidth}
\vspace{2em}
\centering
{\includegraphics{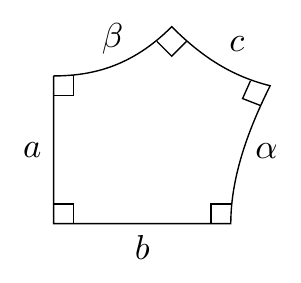}}
\label{fig:pentagon}
\end{minipage}

\subsection{The cross} \label{sec:cross}
We start with a right-angled pentagon $P=P(t)$ with two non-adjacent sides of length $t>0$. Let $\sigma$ be the side between those of length $t$ and let $u$ be the length of each of the other two sides. We have
\begin{equation} \label{eq:sigma}
\cosh \sigma(t) = \coth^2 t = \sinh^2 u(t)
\end{equation}
by Equations \eqref{eq:pentsinh} and \eqref{eq:pentcoth}. Reflect $P$ across the two sides of length $u$ and the vertex opposite to $\sigma$ to obtain a right-angled octagon $O=O(t)$ with side lengths alternating between $2t$ and $\sigma$. Double $O$ across the sides of length $\sigma$ to form a four-holed sphere $C=C(t)$ that we call a \emph{cross}. Each of the four boundary geodesics of $C$ has length $4t$. We refer to them as the left, right, top and bottom boundaries of $C$ following \figref{fig:cross}. Similarly, the cross has a front and a back. 

\begin{figure}[htp]
\centering
\subcaptionbox{The pentagon}
{\includegraphics{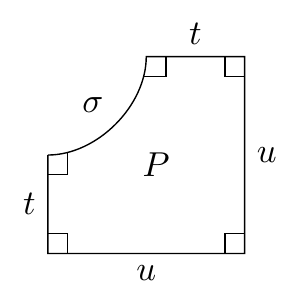}}
\subcaptionbox{The octagon}
{\includegraphics{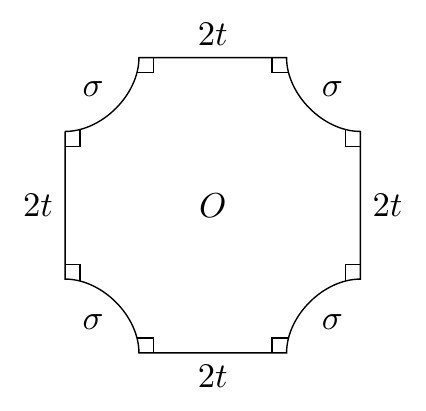}}
\subcaptionbox{The cross}
{\includegraphics{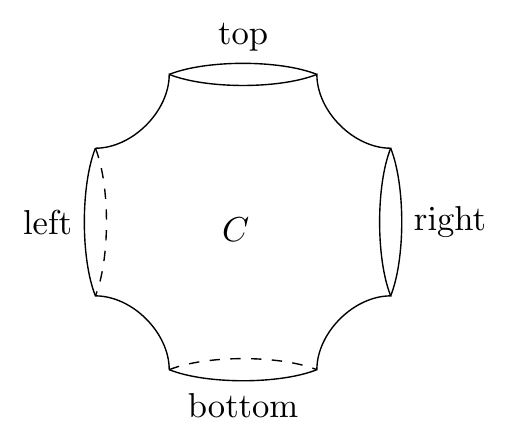}}
\caption{The cross $C$ is made with two octagons, each assembled from four pentagons}\label{fig:cross}
\end{figure}

We note in passing that $C$ is an orbifold cover of a quadrilateral $Q=Q(t)$ with three right angles, one angle equal to $\pi/4$, one side of length $t$ and one side of length $\sigma/2$ obtained by cutting $P$ along the median between $\sigma$ and the opposite vertex. The closed surfaces we construct in the end are also orbifold covers of $Q$, although not regular covers in general.

\subsection{The ring} \label{subsec:the_ring}

Let $n\geq 1$ be an integer. We take a string of $n$ crosses $C_1, C_2, \ldots, C_n$ where the right boundary of $C_j$ is glued to the left boundary of $C_{j+1}$ without twist for $j=1,\ldots,(n-1)$. Finally, the left boundary of $C_1$ is glued to the right boundary of $C_n$ with a half twist (see \figref{fig:ring}). The resulting surface $R=R(n,t)$ is called a \emph{ring}. It is a surface of genus one with $2n$ boundary components.

\begin{figure}[htp]
\centering
{\includegraphics{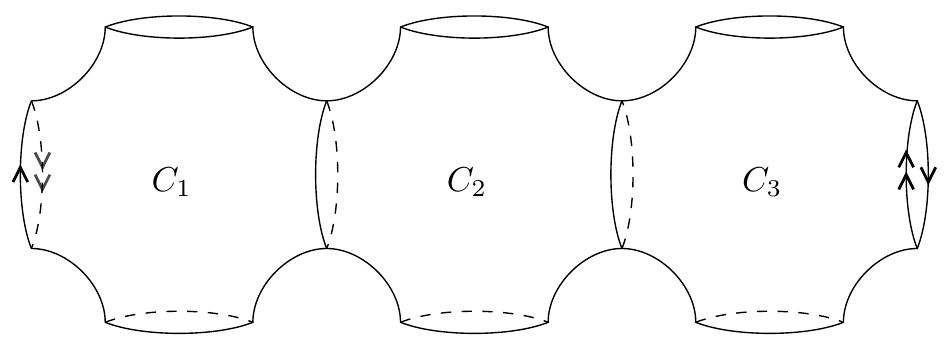}}
\caption{The ring $R$ is a string of $n$ crosses with its ends glued by a half twist}\label{fig:ring}
\end{figure}

There is an alternative description of the ring which is useful for drawing pictures so that no part of the ring is hidden. Take a strip of $2n$ octagons $O_1, \ldots O_{2n}$ with the right side of each glued to the left side of the next and the right side of $O_{2n}$ glued to the left side of $O_1$, forming a topological annulus $A=A(n,t)$. Then the top left and top right sides of $O_j$ are glued to the bottom left and bottom right sides of $O_{n+j}$ respectively for $j=1,\ldots,2n$  in order to form $R$, where indices are taken modulo $2n$ (see \figref{fig:strip}). 

\begin{figure}[htp]
\centering
{\includegraphics{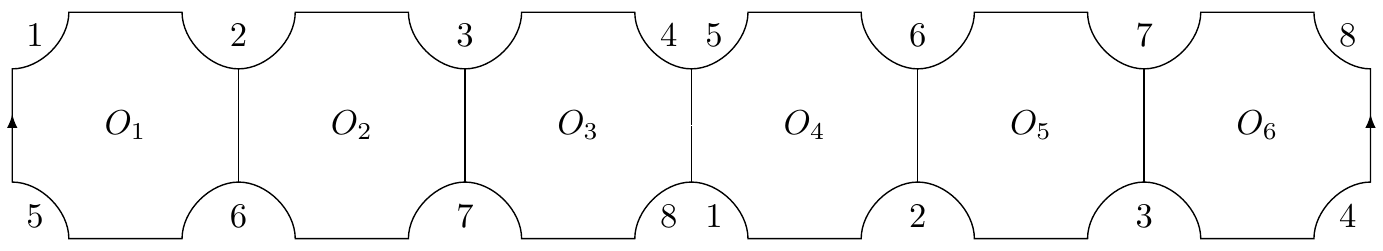}}
\caption{The ring $R$ is also a strip of $2n$ octagons with its left and right sides glued and the segments labelled $\sigma$ identified in pairs in the pattern shown}\label{fig:strip}
\end{figure}

In other words, the sides of the annulus $A$ labelled $\sigma$ are glued in pairs by a glide reflection that reflects across the core geodesic $e$ of $A$ (the horizontal axis of symmetry in \figref{fig:strip}) and translates halfway around $e$. The union of the octagons $O_j$ and $O_{n+j}$ is equal to the cross $C_j$ from the previous description.

\subsection{Geodesics in the ring}

Following Schmutz Schaller, we will often use the same symbol for the name of a curve and its length. The closed geodesics separating adjacent crosses in the ring are called \emph{$f$-curves}. More precisely, for each $j$ from $1$ to $n$, we let $f_j$ be the left boundary of $C_j$. Each $f$-curve has length $4t$. The geodesic that runs along the horizontal axis of symmetry of all the crosses is called $e$, which has length $4n \cdot u$ or
\begin{equation} \label{eq:length_of_e}
e=4n \arcsinh(\coth t). 
\end{equation}

The next geodesics of interest are called \emph{$a$-curves} and \emph{$b$-curves}. For each $j \in \{1,\ldots, 2n\}$, let $a_j$ be the geodesic joining the bottom of the left side of $O_j$ and the top of the left side of $O_{n+j}$ (these two points are identified in $R$) and is otherwise disjoint from the seams and the octagons $O_{n+j}, O_{n+j+1}, ..., O_{n+j+(n-1)}$, where indices are taken modulo $2n$ (see \figref{fig:curves}). Similarly, we let $b_j=\rho_{f_j}(a_j)$ where $\rho_{f_j}:R \to R$ is the reflection across the geodesic $f_j$. By symmetry, all the $a$-curves and $b$-curves have the same length which we denote by $a$.
\begin{figure}[htp]
\centering
{\includegraphics{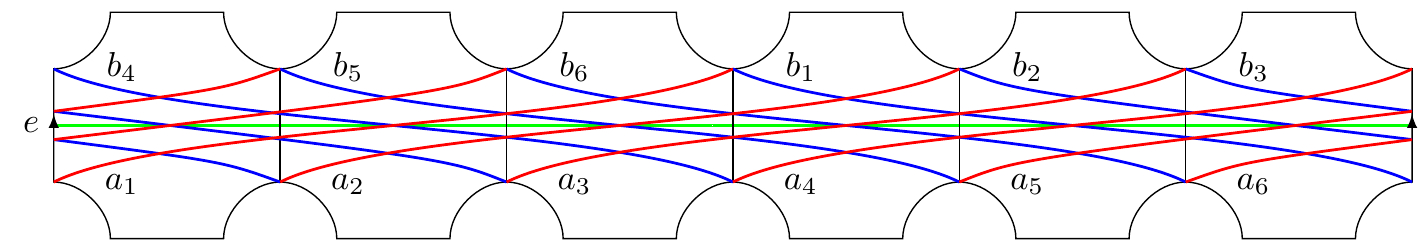}}
\caption{The geodesic $e$ (in green), the $a$-curves (in red), and the $b$-curves (in blue)}\label{fig:curves}
\end{figure}

Observe that $i(a_j,e)=i(b_j,e)=i(a_j,b_j)=1$ for every $j$ and that the curves $a_j$, $b_j$ and $e$ bound two triangles with the same interior angles. These two triangles are therefore congruent, so that their side lengths are $a_j/2$, $b_j/2$ and $e/2$. In particular, they are isoceles since $\ell(a_j)=\ell(b_j)$. The altitude of each triangle has length $t$ and bisects the base, which yields the formula
\begin{equation} \label{eq:length_of_a}
\cosh(a/2) = \cosh(t) \cosh(e/4)
\end{equation}  
by \eqnref{eq:righttriangle} for right triangles. One such pair of triangles is illustrated in \figref{fig:angles}.

\begin{remark}
Surfaces of genus $1$ with $m$ boundary components are studied extensively in \cite{SchmutzGerman} where it is shown that the lengths of the boundary geodesics and the $a$-, $b$- and $e$-curves in such a surface define an injective function on Teichm\"uller space. Actually, the length of any boundary geodesic can be recovered from the remaining ones. This detailed analysis is pursued in \cite[Section 4]{SchmutzMaxima} where these rings serve as building blocks for constructing maximal surfaces. We combine rings differently, resulting in a more flexible construction.
\end{remark}

\subsection{Symmetries of the ring}

There is an orientation-reversing isometric involution $\rho_{\text{seams}} : R \to R$ that has the union of the $\sigma$-segments (the \emph{seams} of the crosses) as its set of fixed points. The map $\rho_{\text{seams}}$ exchanges the front and back octagons in each cross $C_j \subset R$. It acts as a glide reflection along $e$ by half its length. 

Another obvious isometry is the reflection $\rho_e : R \to R$ across the geodesic $e$. This isometry permutes the top and bottom of each cross. 

For each $j \in \{1,\ldots,n\}$, there is a reflection $\rho_{f_j}$ across the geodesic $f_j$.

Another useful isometry $\eta : R \to R$ simply shifts each $O_j$ to $O_{j+1}$, where indices are taken modulo $2n$. That is, $\eta$ is a hyperbolic translation along $e$ to the right by distance $e/2n=2u$. 

Lastly, for each $j \in \{1,\ldots,n\}$ the composition $\nu_j:= \eta \circ \rho_{f_j}$ is the reflection of $R$ across the vertical axis of symmetry of $C_j$.

\subsection{Systoles in the ring}

We now determine the systoles in the ring under certain conditions that depend on $t$. We will later fix a value of $t$ for which these conditions are satisfied.

The following well-known criterion is very useful for finding systoles.

\begin{lem} \label{lem:surgery}
If two closed geodesics $\alpha$ and $\beta$ on a compact hyperbolic surface with geodesic boun\-da\-ry intersect at least twice transversely, then there exists a closed geodesic $\gamma$ of length strictly less than $(\ell(\alpha)+\ell(\beta))/2$. In particular, two distinct systoles can intersect at most once. 
\end{lem}
\begin{proof}
Let $p$ and $q$ be two intersection points of $\alpha$ and $\beta$. Construct a curve $\delta$ by taking the shorter subarc of $\alpha$ between $p$ and $q$ and similary for $\beta$. Since geodesic bigons are non-contractible, $\delta$ is homotopic to a closed geodesic $\gamma$ that is strictly shorter. 
\end{proof}

We will apply the contrapositive of the last sentence in the statement repeatedly: \emph{if two systoles intersect at least twice, then they coincide}. We use this fact in combination with the various symmetries of the ring to determine its systoles. We proceed by elimination, arguing that any geodesic---save for a few exceptions---intersects some of its translates at least twice transversally, hence cannot be a systole in view of the above. 

\begin{prop} \label{prop:systoles_in_ring}
Let $n\geq 1$ and $t>0$. Assume that $a(t) < 4t$ and $a(t) < e(t)$. Then the systoles in $R(n,t)$ are exactly the $a$-curves  and the $b$-curves.
\end{prop}
\begin{proof}
Let $\gamma$ be a systole of $R$. We claim that $\gamma$ intersects the seams, $e$, and each $f$-curve at most once. Otherwise, $\gamma$ and its image $\gamma^*$ by one of the reflections $\rho_{\text{seams}}$, $\rho_e$, or $\rho_{f_j}$ intersect at least twice. In that case $\gamma=\gamma^*$ by \lemref{lem:surgery}. We rule out the possibility that $\gamma$ coincides with $\rho_{\text{seams}}(\gamma)$, $\rho_e(\gamma)$, or $\rho_{f_j}(\gamma)$ one by one below, thereby proving the claim.

Suppose that $\rho_{\text{seams}}(\gamma) = \gamma$ and that $\gamma$ is disjoint from $e$. Then either $\gamma$ is a boundary component of $R$ in which case $\ell(\gamma) = 4t > a$ and $\gamma$ is not a systole, or else $\gamma$ intersects its shift $\eta(\gamma)$ twice transversely, contradicting \lemref{lem:surgery}. We conclude that $\gamma$ intersects $e$, and it does so at least twice by $\rho_\text{seams}$-symmetry. Therefore $\rho_e(\gamma)$ and $\gamma$ intersect at least twice as well so that they coincide. Then either $\gamma=e$ or $\gamma \perp e$. In the first case $\ell(\gamma)>a$ by hypothesis so that $\gamma$ is not a systole. In the second case $\gamma$ has to be equal to some $f$-curve, so that $\ell(\gamma)=4t > a$. We conclude that $\gamma$ intersects the seams at most once. Actually, $\gamma$ intersects the seams exactly once. Indeed, the complement of the seams is a topological annulus whose only simple closed geodesic is $e$, which is not a systole. Thus $\gamma$ cannot be disjoint from the seams.

Now suppose that $\gamma$ intersects $e$ at least twice so that $\rho_e(\gamma)=\gamma$. Since $\rho_e$ does not fix any point on the seams, the number of intersection points between $\gamma$ and the seams is even, which contradicts the previous paragraph. Therefore, $\gamma$ intersects $e$ at most once. In fact, $\gamma$ cannot be disjoint from $e$ either. This is because the seams disconnect $R \setminus e$, yet $\gamma$ intersects them only once. This shows that $\gamma$ intersects $e$ exactly once. 

Lastly, suppose that $\gamma$ intersects some $f_j$ at least twice so that $\rho_{f_j}(\gamma)=\gamma$. Since $\gamma$ cannot be equal to $f_j$, it is orthogonal to it. Moreover, $\gamma$ must intersect the seams and $e$ at one of the places where $f_j$ does, for otherwise there would be a second intersection point by $\rho_{f_j}$-symmetry. The only closed curve that is orthogonal to $f_j$ at one of these four points is $e$, which is too long. Hence $\gamma$ intersects each $f$-curve at most once. 

Now that the claim is proved, it is not hard to show that $\gamma$ is either an $a$-curve or a $b$-curve. If we cut $R$ along the seams, we get an annulus $A$. The curve $\gamma$ gets cut into an arc $\omega$ in $A$ joining a pair of points that get identified by the gluing pattern. The arc $\omega$ must join a point on the bottom boundary of $A$ to a point on the top boundary since it intersects $e$. Moreover, $\omega$ cannot wrap around $A$ more than once, for otherwise it would intersect some $f$-curve twice. Thus $\omega$ wraps exactly halfway around $A$ (remember, the seams are glued via a glide reflection along $e$ by distance $e/2$). It follows that $\gamma$ is homotopic to---hence equal to---one of the $a$-curves or $b$-curves. 
\end{proof}

\subsection{Transverse rings}

Since the crosses used to build the ring $R$ have diagonal symmetry, we can make two rings overlap along a shared cross. We call this configuration a \emph{pair of transverse rings}. We can think of one ring as being horizontal and the other vertical, as in \figref{fig:transverse_rings}. The $e$-curves in the two rings intersect twice, bisecting each other perpendicularly. There are four different ways to apply the surgery procedure from the proof of \lemref{lem:surgery} to this pair of curves, yielding four geodesics shorter than $e$ that we call \emph{$c$-curves}. One of them is depicted in \figref{fig:transverse_rings}. 

\begin{figure}[htp]
\centering
{\includegraphics{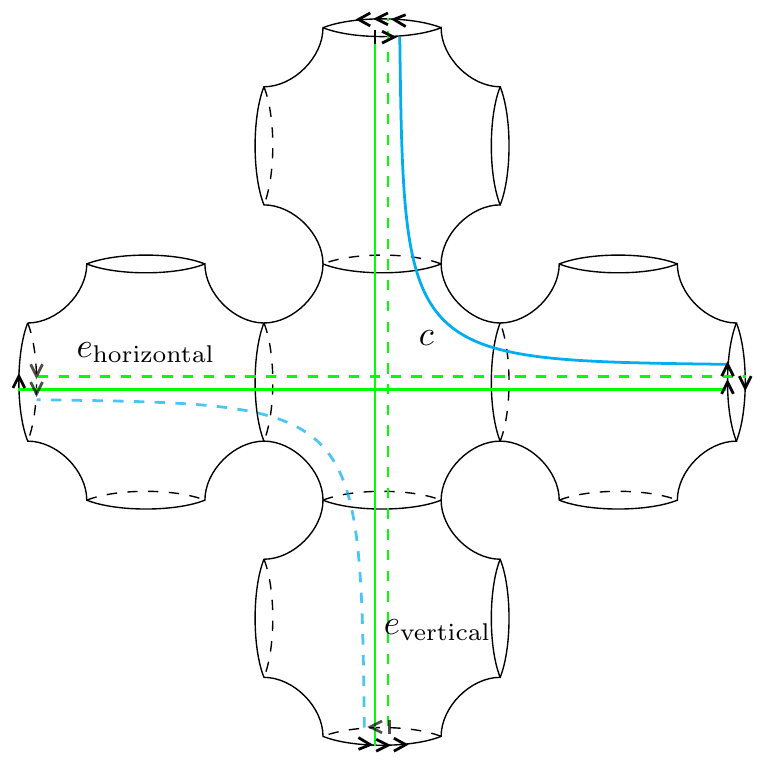}}
\caption{One of the four $c$-curves in a pair of transverse rings, obtained by surgery on the $e$-curves}\label{fig:transverse_rings}
\end{figure}

The four $c$-curves have equal length since they are related by symmetries. Furthermore, there is a right-angled pentagon with two adjacent sides of length $e/4$ and the opposite side of length $c/2$ (see \figref{fig:transverse_rings}). \eqnref{eq:pentsinh} gives the formula
\begin{equation} \label{eq:length_of_c}
\cosh(c/2)=\sinh^2(e/4)
\end{equation}
for the length of any $c$-curve.

When $n=1$ the pair of transverse rings is reduced to a single cross and there are actually only two $c$-curves because some surgeries on the $e$-curves coincide. In this case, each $c$-curve is equal to the union of two opposite seams and \eqnref{eq:length_of_c} is really the same as \eqnref{eq:sigma}. We will analyze this case more carefully in the next subsection.

The next step is to fix the parameter $t$ in such a way that the curves $a$, $b$ and $c$ all have the same length. A first useful observation is that $c$ is a decreasing function of $t$.

\begin{lem} \label{lem:c_decreases}
For every $n \geq 1$, the functions $e(t)$ and $c(t)$ are decreasing in $t$. 
\end{lem}
\begin{proof}
Recall that $e(t)=4n \arcsinh(\coth(t))$. Since $\coth$ is decreasing and $\arcsinh$ is increasing, $e$ is decreasing. Therefore $c(t) = 2 \arccosh(\sinh^2(e(t)/4))$ is decreasing as well, being the composition of a decreasing function with an increasing one.
\end{proof}

We use this to prove the existence and uniqueness of a parameter $t_n$ such that the curves $a$, $b$ and $c$ in the pair of transverse rings all have the same length. 

\begin{lem} \label{lem:existenceanduniqueness}
For every $n \geq 1$, there exists a unique $t_n>0$ such that $a(t_n)=c(t_n)$. 
\end{lem}
\begin{proof}
We have
\[ \frac{\cosh(a(t)/2)}{\cosh(e(t)/4)} = \cosh(t) \quad \text{and} \quad \frac{\cosh(c(t)/2)}{\cosh(e(t)/4)} = \frac{\sinh^2(e(t)/4)}{\cosh(e(t)/4)} = \tanh(e(t)/4) \sinh(e(t)/4) \]
by Equations \eqref{eq:length_of_a} and \eqref{eq:length_of_c}. Therefore, the equation $a(t)=c(t)$ is equivalent to
\begin{equation} \label{eq:equivalence}
\cosh(t) = \tanh(e(t)/4) \sinh(e(t)/4).
\end{equation}
The left-hand side of \eqref{eq:equivalence} is an increasing function of $t$ which diverges as $t \to \infty$. The right-hand side is decreasing in $t$ since it is the product of two positive decreasing functions. Moreover, it diverges as $t \to 0$ since $e(t)$ does. The existence and uniqueness of $t_n$ follows. 
\end{proof}

From now on, we will only work with the rings $R(n,t_n)$ with $t_n$ as in \lemref{lem:existenceanduniqueness}. In order to determine the systoles in that ring, we need to check that the hypotheses of \propref{prop:systoles_in_ring} are satisfied, but this is only true when $n \geq 2$. The case $n=1$ is treated separately in the next subsection.

\begin{figure}[htp]
\centering
{\includegraphics[scale=.6]{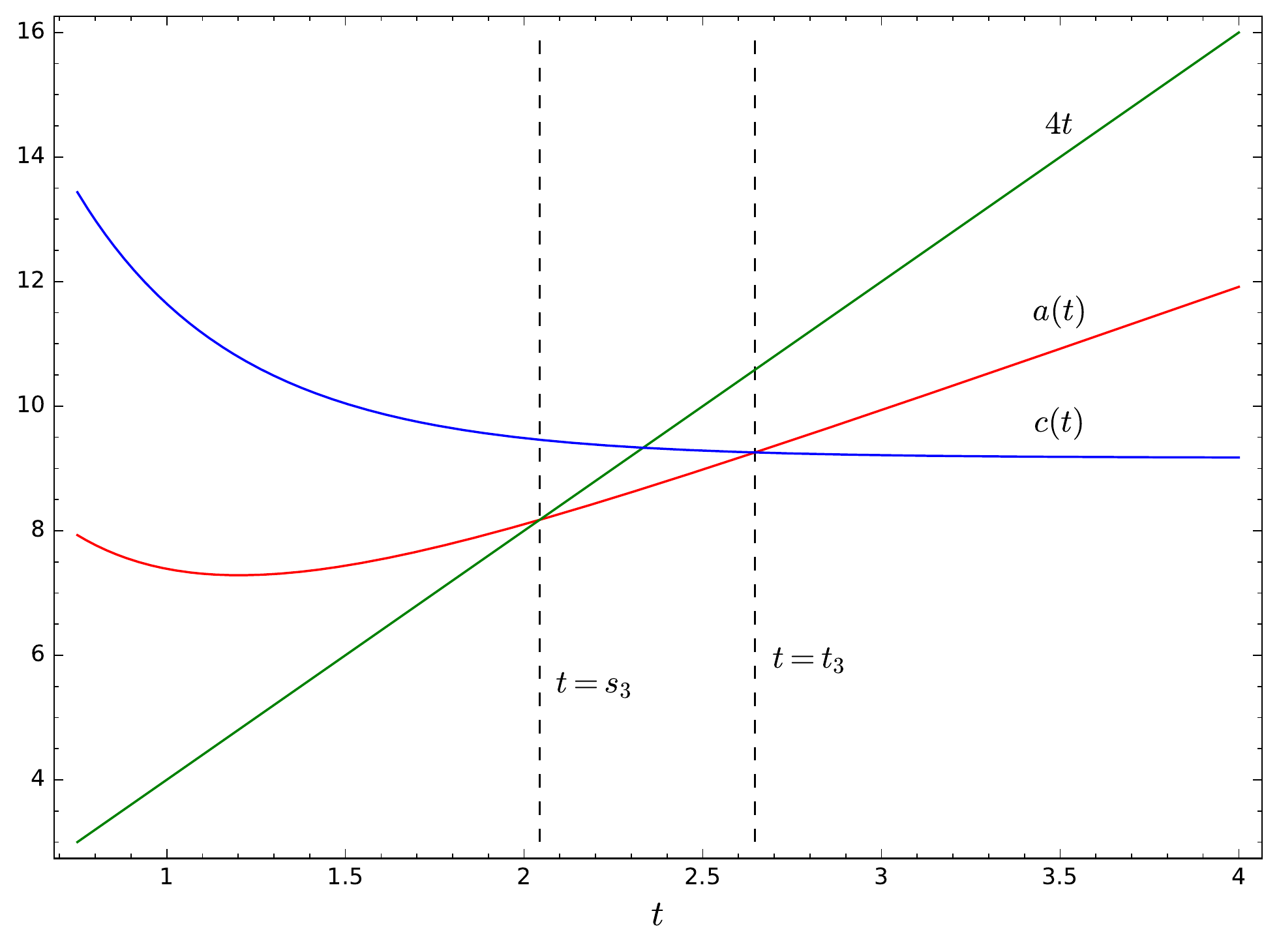}}
\caption{A plot of the functions $a(t)$, $c(t)$ and $4t$ for $n=3$}\label{fig:plot}
\end{figure}

\begin{lem} \label{lem:a_shorter_than_f}
We have $a(t_n) < 4 t_n$ and $a(t_n) < e(t_n)$ for every $n \geq 2$.
\end{lem}
\begin{proof}
The inequality $a(t_n)=c(t_n)<e(t_n)$ follows from the fact that $c$ is obtained by surgery on two $e$-curves, or can be deduced from \eqnref{eq:length_of_c}.

To show that $a(t_n)<4t_n$ we consider the time $s_n>0$ such that $a(s_n)=4s_n$ and prove that $c(s_n) > 4 s_n$. This implies that $s_n < t_n$ since $c$ is decreasing whereas $a(t)$ diverges as $t \to \infty$. The inequality $a(t_n) < 4 t_n$ then follows from the fact that
\[
\frac{\cosh(a(t)/2)}{\cosh(2t)} = \frac{\cosh(t)}{\cosh(2t)} \cdot \cosh(e(t)/4)
\] 
is decreasing, being the product of two positive decreasing functions. \figref{fig:plot} illustrates this phenomenon for $n=3$.

Hence let $s_n>0$ be the unique parameter such that $a(s_n)=4s_n$. Then 
\[\cosh(s_n)\cosh(e(s_n)/4) = \cosh(a(s_n)/2) = \cosh(2s_n) = 2 \cosh^2(s_n) - 1\]
and 
\begin{align*}
\cosh(c(s_n)/2) = \sinh^2(e(s_n)/4) &= \cosh^2(e(s_n)/4) - 1\\
&= \left( \frac{2 \cosh^2(s_n) - 1}{\cosh(s_n)} \right)^2 - 1.
\end{align*}
Let $x = \cosh^2(s_n)$ so that $\cosh(2s_n) = 2x - 1$ and
\[
\cosh(c(s_n)/2) = \frac{(2x-1)^2}{x} - 1.
\]
The inequality we want to prove is $\cosh(c(s_n)/2) > \cosh(2s_n)$, which is equivalent to $(2x-1)^2 > 2x^2$ or $x > 1 + \frac{1}{\sqrt{2}}$ after simplification. Therefore, it suffices to show that 
\[s_n > \arccosh\left(\sqrt{1 + \frac{1}{\sqrt{2}}}\right) \approx 0.764.\]

But at $t=1$ we get
\begin{align*}
\frac{\cosh(a(1)/2)}{\cosh(2\cdot 1)}&  = \frac{\cosh(1)}{\cosh(2\cdot 1)} \cdot \cosh(e(1)/4) \\
 &= \frac{\cosh(1)}{\cosh(2)} \cdot \cosh(n \arcsinh(\coth(1)))\\
 & \geq \frac{\cosh(1)}{\cosh(2)} \cdot \cosh(2 \cdot 1.086) > \cosh(1) > 1
\end{align*}
which implies that $s_n>1$ and finishes the proof.
\end{proof}

We conclude that the systoles in the ring $R(n,t_n)$ are the $a$-curves and the $b$-curves.

\begin{cor} \label{cor:systoles_fixed_t}
For every $n\geq 2$, the systoles in $R(n,t_n)$ are the $a$-curves and the $b$-curves.
\end{cor}
\begin{proof}
This follows from \lemref{lem:a_shorter_than_f} and \propref{prop:systoles_in_ring}.
\end{proof}

We observed earlier that $c$ is a decreasing function of $t$. The function $a$ is not monotone but we can show it is increasing at $t_n$. These two facts will play a key role in \secref{sec:minimality}.

\begin{lem} \label{lem:a_increases}
We have $a'(t_n) > 0$ for every $n \geq 2$.
\end{lem}
\begin{proof}
From $\cosh(a(t)/2)=\cosh(t)\cosh(e(t)/4)$ we compute
\begin{align*}
\sinh(a(t)/2)\, a'(t)/2 &= \sinh(t)\cosh(e(t)/4) +\cosh(t) \sinh(e(t)/4)\, e'(t)/4 \\
&> \sinh(e(t)/4) \left[ \sinh(t) + \cosh(t)\,e'(t)/4 \right].
\end{align*}
Thus it suffices to show that $-e'(t_n)/4 < \tanh(t_n)$. Since $e(t)/4 = n \arcsinh(\coth t)$ we get
\[
-e'(t)/4 = \frac{n}{\sinh^2(t)\sqrt{\coth^2(t)+1}} < \frac{n}{\sinh^2(t)\,\sqrt{2}}
\]
so that the required inequality becomes $n < \sqrt{2}\,\tanh(t_n)\sinh^2(t_n)$. 

We know that $t_n>1$ from the proof of \lemref{lem:a_shorter_than_f}. Furthermore,  one can show that
\[ \sqrt{2}\, \tanh(x)\sinh^2(x) > 0.963 \cdot \cosh(x) \] 
for every $x \geq 1$. Indeed, $\sinh^3(x)/\cosh^2(x)$ is increasing and the inequality can be verified numerically at $x=1$. Recall that $\cosh(t_n) = \tanh(e(t_n)/4)\sinh(e(t_n)/4)$ by definition of $t_n$. We thus obtain
\begin{align*}
\sqrt{2}\, \tanh(t_n)\sinh^2(t_n) &> 0.963 \cdot \cosh(t_n) \\
& = 0.963 \cdot \tanh(e(t_n)/4)\sinh(e(t_n)/4) \\
& > 0.963 \cdot \tanh(n \lambda )\sinh(n \lambda) \geq n
\end{align*}
for every $n\geq 2$, where $\lambda = \arcsinh(1)$. The last inequality holds because the function $\tanh(\lambda x) \sinh(\lambda x)/x$ is increasing in $x$ and larger than $1/0.963$ at $x=2$. This implies the desired result.
\end{proof}

In addition to knowing the systoles in the ring, we also need an estimate on the lengths of arcs that enter and exit the ring from a given cross. Since the arcs going vertically across any cross $C_j \subset R$ are fairly short, we need to exclude them. 

\begin{lem} \label{lem:arcs_in_ring}
Let $n \geq 2$. Any non-trivial arc in $R(n,t_n)$ from one boundary component to itself is longer than $a(t_n)/2$. Any geodesic arc that joins the top and bottom of a cross $C_j \subset R(n,t_n)$ but is not contained in $C_j$ is longer than $a(t_n)/2$. 
\end{lem}
\begin{proof}
Let $\gamma$ be a shortest non-trivial arc from one boundary $B$ of $R(n,t_n)$ to itself. In particular, $\gamma$ is geodesic and orthogonal to the boundary. 

If $\gamma$ intersects some $f_j$ twice, then we can reflect a subarc $\omega \subset \gamma$ from $f_j$ to itself across $f_j$ to obtain a non-trivial closed curve of length $2\ell(\omega)$ in $R(n,t_n)$. By \corref{cor:systoles_fixed_t} we get that $2\ell(\gamma) > 2 \ell(\omega)  \geq a(t_n)$.

If $\gamma$ intersects the seams, then we can perform a surgery on $\gamma$ and $\rho_{\text{seams}}(\gamma)$ to obtain a strictly shorter essential arc from $B$ to itself, unless $\gamma=\rho_{\text{seams}}(\gamma)$. One way to see this is to double the ring $R(n,t_n)$ across its boundary and apply \lemref{lem:surgery} to the doubled arcs. Thus distinct non-trivial arcs of minimal length from $B$ to itself are disjoint. But if $\gamma=\rho_{\text{seams}}(\gamma)$, then $\gamma$ intersects some $f$-curve at least twice, hence is longer than $a(t_n)/2$ by the previous paragraph. The only exception is if $\gamma$ is contained in a single cross $C_j$. But in that case, if we double $C_j$ across $B$ we obtain a pair of crosses and a closed geodesic of length $2\ell(\gamma)$ in it. This pair embeds isometrically in $R(n,t_n)$, showing that $2\ell(\gamma) > a(t_n)$. The inequality is strict because no systole in $R(n,t_n)$ is symmetric about any $f$-curve. 

We can therefore assume that $\gamma$ is disjoint from the seams and intersects each $f$-curve at most once. Up to the symmetries of $R(n,t_n)$, this leaves two possibilities for $\gamma$ depending whether it intersects $e$ or not. 

Recall that the complement of the seams in $R(n,t_n)$ is an annulus $A$. As such, there is a well-defined orthogonal projection $A \to e$. If $\gamma$ does not intersect $e$, then it intersects all the $f$-curves, and its orthogonal projection onto $e$ is longer than 
\[
\frac{(2n-1)}{2n} e(t_n) > \frac12 e(t_n) > \frac12 a(t_n).
\] 
See \figref{fig:proj}. Since the orthogonal projection does not increase distances, we get that $\ell(\gamma) > a(t_n)/2$.

\begin{figure}[htp]
\centering
{\includegraphics{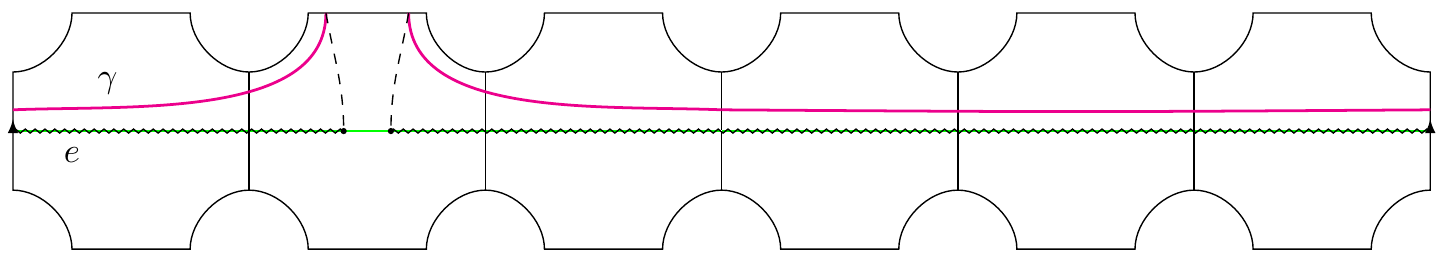}}
\caption{If $\gamma$ does not intersect $e$, then its projection onto $e$ covers most of $e$}\label{fig:proj}
\end{figure}

If $\gamma$ intersects $e$, then $\gamma$ and $\rho_e(\gamma)$ intersect. One of the two possible surgeries on $\gamma \cup \rho_e(\gamma)$ yields a pair of arcs $\alpha$ and $\beta$, each joining the top and bottom boundaries of some cross $C_j$ in $R(n,t_n)$, neither of which can be homotoped into $C_j$ (see \figref{fig:arcs}). This gives $\ell(\gamma) >\ell(\alpha)=\ell(\beta)$, so it suffices to show that $\ell(\alpha)>a(t_n)/2$. We have reduced the first part of the statement of the lemma to the second part.

\begin{figure}[htp]
\centering
{\includegraphics{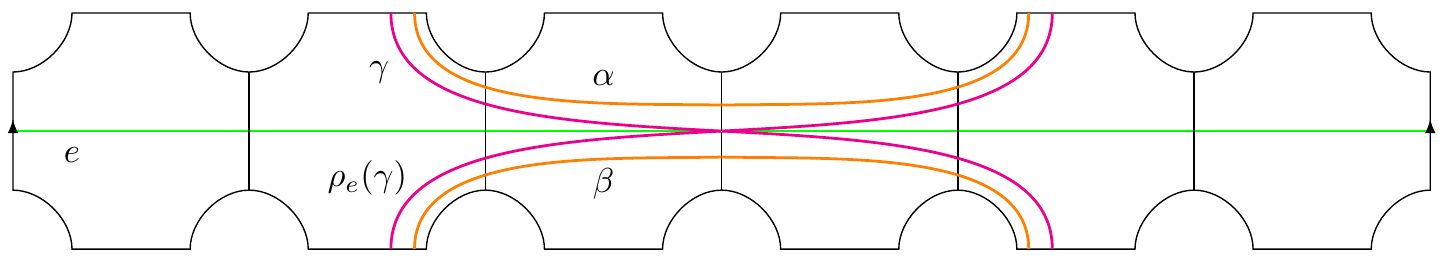}}
\caption{If $\gamma$ intersects $e$ there is a surgery on $\gamma \cup \rho_e(\gamma)$ producing a pair of arcs $\alpha$ and $\beta$ joining two opposite boundaries of a cross}\label{fig:arcs}
\end{figure}

Let $\tau$ be an arc of minimal length in $R(n,t_n)$ that joins the top and bottom boundaries of some cross $C_j$ and cannot be homotoped into $C_j$. By the same argument as above, we may assume that $\tau$ intersects each $f$-curve at most once and is disjoint from the seams. If $\tau$ intersects $e$, then it wraps most of the way around the annulus $A$ so that its orthogonal projection onto $e$ is longer than $e(t_n)/2 > a(t_n)/2$ similarly as above. Otherwise, $\tau$ is equal to the arc $\alpha$ from the preious paragraph or one of its images by the group $\langle \rho_{\text{seams}}, \rho_e, \nu_j \rangle$ where $\nu_j$ is the reflection swapping the left and right sides of $C_j$. In any case, there is a right-angled pentagon with two adjacent sides of lengths $e/4$ and $e/4n$, and the opposite side of length $\tau/2$ (see \figref{fig:arc_estimate}). Equations \eqref{eq:pentsinh} and \eqref{eq:length_of_e} give
\[
\cosh(\tau/2)=\sinh(e/4)\sinh(e/4n) = \sinh(e/4)\coth(t)> \sinh(e/4).
\]
Squaring yields
\[
\frac{\cosh(\tau)+1}{2} = \cosh^2(\tau/2) > \sinh^2(e/4) = \cosh(c/2)
\]
hence
\[
\cosh(\tau) > 2\cosh(c/2)-1 > \cosh(c/2).
\]
This shows that $\ell(\tau) > c(t_n)/2 = a(t_n)/2$, which concludes the proof.
\end{proof}

\begin{figure}[htp]
\centering
{\includegraphics{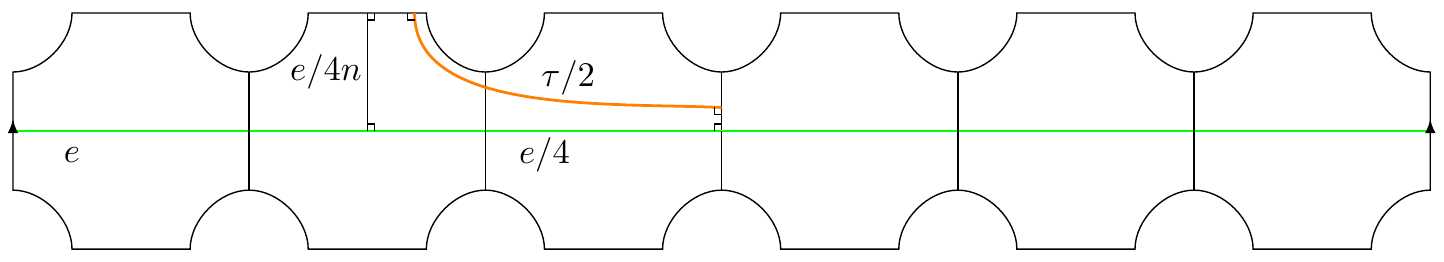}}
\caption{The right-angled pentagon allowing us to compute the length of the shortest arc in \lemref{lem:arcs_in_ring}}\label{fig:arc_estimate}
\end{figure}

\subsection{The Bolza surface}

When $n=1$, the pair of transverse rings is a closed surface of genus $2$ obtained by gluing the opposite sides of the cross $C(t_1)$ with half twists. We now show that this surface--- denoted $\tree(1)$---is the \emph{Bolza surface}, which is the surface of genus two with largest automorphism group (cf. \cite[p.588]{SchmutzMaxima}).

\begin{thm} \label{thm:bolza}
$\tree(1)$ is the Bolza surface.
\end{thm}
\begin{proof}
Let $s$ be the side length of a regular hyperbolic triangle with interior angles $\pi/4$. Eight such triangles fit together at a point to form a regular right-angled octagon $O$. Glue two such octagons together to form a cross isometric to $C(s/2)$, then glue opposite ends of $C(s/2)$ with half twists. The $a$- and $b$-curves in the resulting closed surface are main diagonals of $O$, hence have length $2s$. Similarly, each $c$-curve is equal to the union of two opposite sides of the octagon, hence has length $2s$. This shows that $a(s/2)=2s=c(s/2)$ so that $t_1=s/2$. In particular, the $f$-curves in $\Sigma(1)$ have the same length as the curves of type $a$, $b$ and $c$.

Now cut the front octagon of $\Sigma(1)$ into $8$ equilateral triangles and attach them to the corresponding sides of the back octagon. The result is a regular octagon with interior angles $\pi/4$. The sides of the latter are identified in opposite pairs to form $\tree(1)$ (see \figref{fig:bolza}). This is a standard representation of the Bolza surface \cite[Section 3]{Karcher}.
\end{proof}

\begin{figure}[htp]
\centering
{\includegraphics{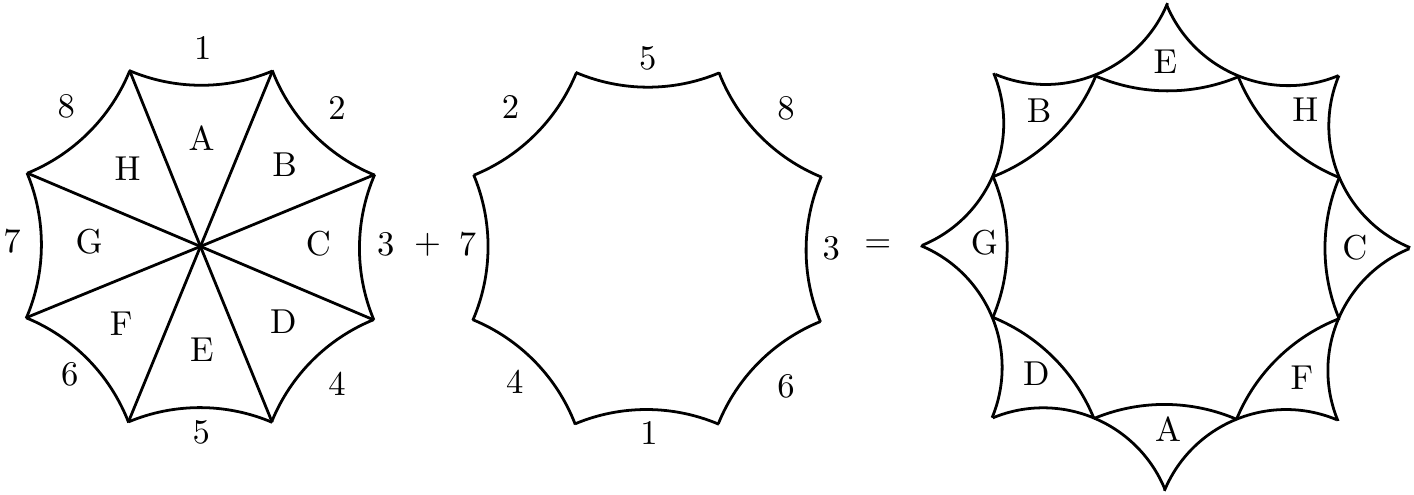}}
\caption{$\Sigma(1)$ is the Bolza surface}\label{fig:bolza}
\end{figure}

\begin{remark} \label{rem:bolza_length}
The above proof shows that $a(t_1)=c(t_1)=4t_1$. After some elementary algebraic manipulations\footnote{We have $\cosh(2t_1)=\cosh(c(t_1)/2)=\sinh^2(e(t_1)/4)=\coth^2(t_1)$ by definition, which implies that $(2\cosh^2(t_1) -1)(\cosh^2(t_1)-1) = \cosh^2(t_1)$. This is a quadratic equation in $\cosh^2(t_1)$ whose only solution larger than $1$ is given by $\cosh^2(t_1) = 1 + 1 /\sqrt{2}$. Thus $\cosh(2t_1)= 2 \cosh^2(t_1) - 1 = 1 + \sqrt{2}$.}, one arrives at the exact formula $t_1 = \arccosh(1+\sqrt{2})/2$.
\end{remark}

\subsection{The tree of rings} \label{subsec:tree}

For $n \geq 1$, let $T(n)$ be the $n$-regular tree. We build a hyperbolic surface $\tree(n)$ called the \emph{tree of rings} by replacing each vertex $v \in T(n)$ with a copy $R_v$ of the ring $R(n,t_n)$ such that two rings $R_v$ and $R_w$ are transverse if and only if the vertices $v$ and $w$ are adjacent in $T(n)$. In other words, each edge of $T(n)$ is replaced by a cross $C(t_n)$ and the crosses are glued in such a way that those corresponding to the $n$ edges adjacent to any vertex in $T(n)$ form a ring isometric to $R(n,t_n)$. 

The resulting surface $\tree(1)$ is closed of genus two, $\tree(2)$ has two ends accumulated by genus and $\tree(n)$ has a Cantor set of ends accumulated by genus when $n\geq 3$ (see \figref{fig:tree}).

\begin{figure}[htp]
\centering
{\includegraphics[width=\textwidth]{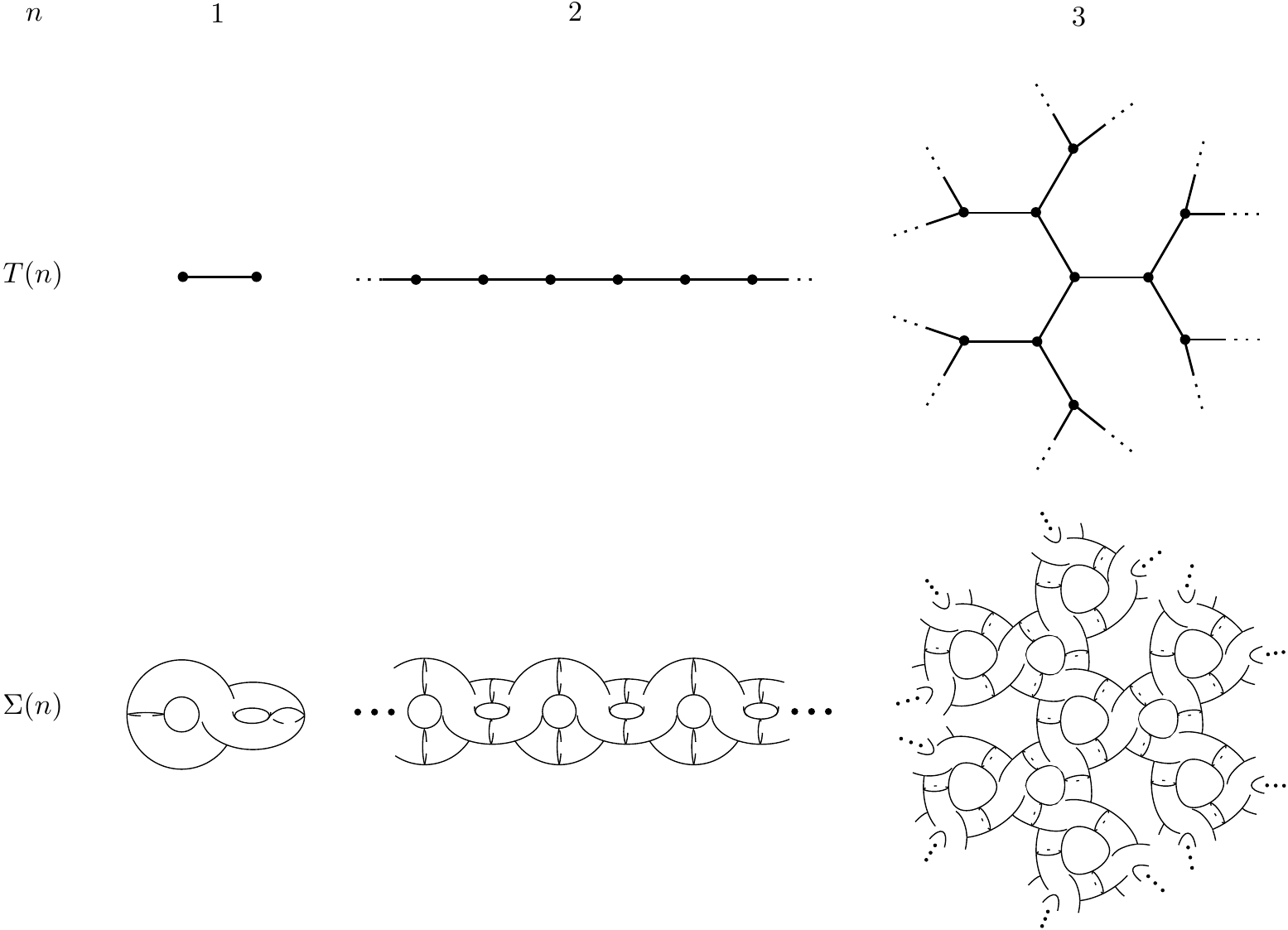}}
\caption{The tree of rings for $n=1,2,3$}\label{fig:tree}
\end{figure}

We now determine the systoles in $\tree(n)$, starting with $\tree(1)$ as a warm-up.

\begin{prop} \label{prop:systoles_bolza}
The systoles in $\tree(1)$ are the $a$- and $b$-curves in the horizontal and vertical rings, the two $c$-curves and the two $f$-curves. The total number of systoles is $12$ and their length is $2\arccosh\left(1+\sqrt{2}\right) \approx 3.057$.
\end{prop}
\begin{proof}
The $e$-curves are longer than the $c$-curves by construction, hence longer than the $a$- and $b$-curves. The proof of \propref{prop:systoles_in_ring} applies almost verbatim to show that the shortest curves disjoint from the horizontal (resp. vertical) $f$-curve are the $a$- and $b$-curves in the horizontal (resp. vertical) ring together with the vertical (resp. horizontal) $f$-curve. The only difference is that the $f$-curves were ruled out in \propref{prop:systoles_in_ring} for being too long by hypothesis. 

Let $\gamma \subset \tree(1)$ be a systole that intersects both $f$-curves. Consider the two diagonal axes of symmetry of the cross $C(t_1)$. These curves divide $\tree(1)$ into a union of two congruent annuli with piecewise geodesic boundary, each containing one of the $f$-curves as its core geodesic. By hypothesis $\gamma$ traverses each annulus at least once. It is easy to see that the shortest arc across either annulus has length $c/2$. Therefore $\ell(\gamma)\geq c$ with equality if and only if $\gamma$ is a concatenation of two seams, i.e., a $c$-curve.

Since the $a$-, $b$-, $c$- and $f$-curves all have the same length equal to $2\arccosh\left(1+\sqrt{2}\right)$ (see \remref{rem:bolza_length}), they are the systoles.
\end{proof}

\begin{prop} \label{prop:systoles_in_tree}
For every $n\geq 2$, the systoles in the tree of rings $\tree(n)$ are the $a$- and $b$-curves contained in rings, together with the $c$-curves contained in pairs of transverse rings.
\end{prop}
\begin{proof}
Let $\gamma$ be a systole of $\tree(n)$. We define the \emph{shadow} of $\gamma$ in $T(n)$ as follows. First we cut $\gamma$ along the $f$-curves into subarcs $\gamma_1, \ldots, \gamma_k$ labelled in cyclic order along $\gamma$. For each subarc $\gamma_j$ that joins two boundaries of a cross $C$ which are not opposite of each other (i.e., each subarc that ``turns'' from one ring to another), its shadow $s(\gamma_j)$ is the edge in $T(n)$ corresponding to the pair of transverse rings that intersect along $C$. The shadow $s(\gamma_j)$ of each subarc $\gamma_j$ that does not turn is defined to be the vertex $v \in T(n)$ corresponding to the ring $R_v$ containing $\gamma_j$ in its interior. The shadow $s(\gamma)$ is defined as the concatenation of the shadows $s(\gamma_1), \ldots, s(\gamma_k)$. This forms a loop in $T(n)$.
 
The shadow $s(\gamma)$ is not well-defined if $\gamma$ is disjoint from the $f$-curves or is equal to one of them. But in that case $\gamma$ is contained in a ring so that it is either an $a$-curve or a $b$-curve by \corref{cor:systoles_fixed_t}.

Being a loop in a tree, $s(\gamma)$ has at least two places where it backtracks, that is, an edge which it traverses twice in a row in opposite directions. By definition of the shadow, a backtrack corresponds to an arc entering and leaving a ring through the same cross, turning at the beginning and at the end. By \lemref{lem:arcs_in_ring}, such an arc is longer than $a(t_n) /2$. In particular, if $s(\gamma)$ has two backtracks happening along two distinct edges, then $\gamma$ has two disjoint subarcs longer that $a(t_n) /2$ each, so that it is not a systole.

This leaves the possibility that $s(\gamma)$ is just a loop formed by traversing one edge $\{v,w\}$ of $T(n)$ twice in opposite directions. In that case, $\gamma$ is contained in a pair of transverse rings $R_v \cup R_w$ and turns exactly twice in the cross $C=R_v \cap R_w$. 

We can write $\gamma$ as the concatenation of two arcs $\gamma_v$ and $\gamma_w$ where $\gamma_v = \gamma \cap R_v$ and $\gamma_w = \gamma \setminus \gamma_v$. This means that $\gamma_v$ contains both turns of $\gamma$. In particular, $\gamma_v$ is not contained in $C$ so that $\ell(\gamma_v) > a(t_n) / 2$ by \lemref{lem:arcs_in_ring}.

Suppose that the two endpoints of $\gamma_v$ belong to the same boundary component of $C$. Then $\gamma_w$---which is contained in $R_w$---can be reflected across that $f$-curve to form a non-trivial closed curve in $R_w$. That curve is longer than $a(t_n)$ by \corref{cor:systoles_fixed_t}, hence $\ell(\gamma_w) > a(t_n) / 2$. This gives $\ell(\gamma)=\ell(\gamma_v)+\ell(\gamma_w)>a(t_n)$. 

By exchanging the roles of $R_v$ and $R_w$, the previous argument shows that the two turning subarcs of $\gamma$ have endpoints in all four boundary components of $C$. This implies that $\gamma$ intersects one of the two diagonal axes of symmetry of $C$---call it $d$---twice. But the reflection of $C$ in the curve $d$ extends to a global isometry $\rho_d$ of $\tree(n)$. By \lemref{lem:surgery}, we have $\rho_d(\gamma)=\gamma$. If $\gamma$ also intersects the seams, then it intersects them twice by symmetry across $d$.  In that case, $\gamma$ is invariant under $\rho_\text{seams}$ as well. But then the two turning subarcs of $\gamma$ in $C$ are mirror images across the seams, hence have endpoints in only two boundary components of $C$. That contradicts the first sentence of this paragraph.

We know that $\rho_d(\gamma)=\gamma$ and that $\gamma$ is disjoint from the seams. Consider the subarc $\alpha \subset \gamma$ contained in $R_v$ with two endpoints on $d$ and let $\beta=\rho_d(\alpha)$ so that $\gamma=\alpha \cup \beta$. If $\alpha$ intersects any $f$-curve twice, then $\ell(\alpha) > a(t_n) /2$ by an argument above so that $\ell(\gamma) = \ell(\alpha) + \ell(\beta)= 2 \ell(\alpha) > a(t_n)$. Thus $\alpha$ intersects each $f$-curve at most once. This determines the homotopy class of $\alpha$ up to moving the endpoints along $d$ since the complement of the seams in $R_u$ is an annulus. That is, $\alpha$ wraps once around $R_u$ intersecting each $f$-curve once along the way, while staying disjoint from the seams. We conclude that $\gamma$ is homotopic to a $c$-curve, hence equal to one of them.
\end{proof}

\subsection{Signed graphs} \label{subsec:signed_graph}

Let $n\geq 3$. In order to get a closed surface, we glue copies of the cross $C(t_n)$ along a finite $n$-regular graph $\Gamma$ instead of the tree $T(n)$. In order to determine the gluings precisely, we need  a bit more structure on $\Gamma$, namely,
\begin{itemize}
\item a cyclic ordering of the edges adjacent to any vertex;
\item a sign $\eps(e_1,e_2) \in \{+,-\}$ attributed to any two consecutive edges $e_1, e_2$ around a vertex, subject to the condition that the product of the signs around any vertex is negative. 
\end{itemize}
We call a graph equipped with this additional structure a \emph{signed graph}. Note that a choice of cyclic ordering around each vertex is known as an (oriented) ribbon structure. However, we will now define when two signed graphs are isomorphic, and such isomorphisms need not preserve the ribbon structure.

Given a vertex $x$ in a signed graph $\Gamma$, we define the \emph{vertex flip} around $x$  to be the operation that reverses the cyclic ordering around $x$ and changes the signs between each edge $e$ containing $x$ and its two immediate neighbors around the vertex $e \setminus x$ (see \figref{fig:flip}). Clearly, any two vertex flips commute. We say that two signed graphs are \emph{isomorphic} if one can be obtained from the other by a set of vertex flips.

\begin{figure}[htp]
\centering
{\includegraphics{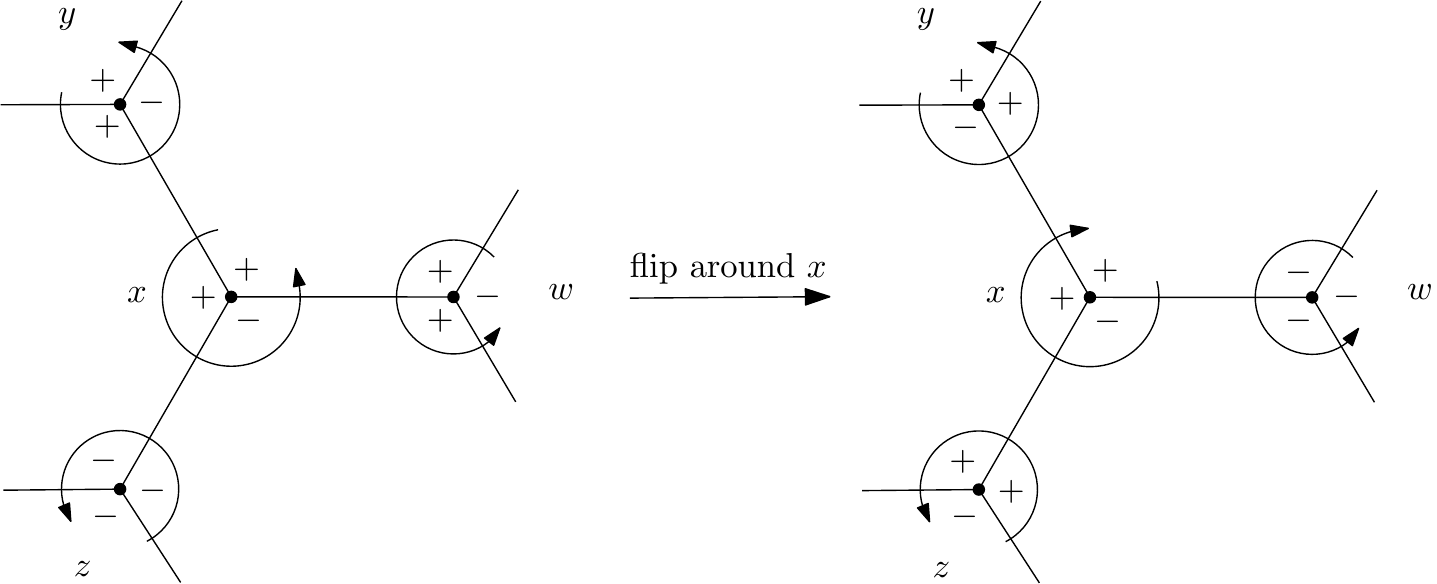}}
\caption{A vertex flip on a signed graph}\label{fig:flip}
\end{figure}

\subsection{Gluing crosses according to a signed graph} \label{subsec:graph_to_surface}

Let $n\geq 3$ and let $\Gamma$ be a connected, $n$-regular, signed graph. We construct a surface $X(\Gamma)$ modelled on $\Gamma$ as follows. To each edge $e$ in $\Gamma$ corresponds a cross $C_e$ isometric to $C(t_n)$. The edge $e=\{u,v\}$ has two neighboring edges (which coincide when $n=2$) around each of $u$ and $v$. We glue 
\begin{itemize}
\item the predecessor of $e$ around $u$ to the left of $C_e$; 
\item  the successor of $e$ around $u$ to the right of $C_e$; 
\item the predecessor of $e$ around $v$ to the bottom of $C_e$; 
\item the successor of $e$ around $v$ to the top of $C_e$. 
\end{itemize}
Each of these gluings is done as to make the seams match. This still leaves two possibilities for each gluing: either with a half twist or not. This is determined using the signs between consecutive edges: the ``$+$" signs mean no twist and the ``$-$" signs call for half twists.

Note that for a string of crosses, the half twists do not affect the isometry type. However, when we close up the string to form a loop, they do. For instance, with an even number of half twists the seams separate, but with an odd number of half twists they do not. It is easy to see that a chain of $n$ crosses isometric to $C(t_n)$ glued end to end is isometric to the ring $R(n,t_n)$ if and only if the number of half twists is odd. This is why we require the product of the signs around each vertex in $\Gamma$ to be negative.   

We also remark that rotating each cross by angle $\pi$ around one of its diagonals exchanges left and bottom as well as right and top. Thus changing the order between $u$ and $v$ above merely switches the horizontal and vertical axes but not the gluings themselves. Each ring can be seen as either horizontal or vertical interchangeably; this notion need not be globally defined.

The surface $X(\Gamma)$ is defined as
\[
X(\Gamma) = \left(\bigsqcup_{e \in E} C_e \right) / \sim
\]
where $E$ is the set of edges of $\Gamma$ and the equivalence $\sim$ identifies boundary points of different crosses as described above. 

The sign structure of $\Gamma$ induces a cyclic ordering of the crosses in each ring. For any ring $R$ in $X(\Gamma)$, there are exactly $n$ other rings transverse to it. When two of these transverse rings pass through adjacent crosses of $R$, let us say that they are \emph{parallel}. Whether the cyclic orderings in parallel rings passing through adjacent crosses  $C_{e_1}$ and $C_{e_2}$ agree or not is indicated by the sign $\eps(e_1,e_2)$. If we reverse the cyclic ordering of the crosses in $R$, then the comparison between parallel rings transverse to $R$ is unaffected. However, for every ring parallel to $R$, there is a change of sign: if orders agreed before, they do not anymore and vice versa. In other words, if $\Gamma_1$ and $\Gamma_2$ are isomorphic signed graphs, then there is an orientation-preserving isometry between $X(\Gamma_1)$ and $X(\Gamma_2)$.

As an example, if $\Gamma$ is the $n$-regular tree $T(n)$ equipped with an arbitrary sign structure, then $X(\Gamma)$ is isomorphic to the tree of rings $\tree(n)$. This is because any two sign structures on $T(n)$ are isomorphic, a fact left as an exercise\footnote{Hint: First show that any sign pattern (with negative product) around a vertex $v$ can be changed into any other (with negative product) by doing some vertex flips around the neighbors of $v$. Furthermore, this can be done even if one neighbor of $v$ is required to be left intact.} to the reader. 

\subsection{The ribbon graph induced by a signed graph}

There is a useful combinatorial object $\widehat{\Gamma}$ that comes between the signed graph $\Gamma$ and the surface $X(\Gamma)$ which makes the correspondence more transparent. This object is a (non-orientable) $4$-regular ribbon graph, and is obtained from $\Gamma$ as follows:
\begin{itemize}
\item for each edge $e=\{u,v\}$ of $\Gamma$ corresponds a vertex $\widehat{e}$ in $\widehat\Gamma$;
\item each vertex in $\widehat\Gamma$ is $4$-valent, and its adjacent edges are given a cyclic order;
\item the vertices in $\widehat\Gamma$ that correspond to the predecessor and successor of $e$ around $u$ in $\Gamma$ share edges with $\widehat{e}$, and these edges are to be opposite in the cyclic order;
\item similarly for the vertices corresponding to the two immediate neighbors of $e$ in the cyclic order around $v$;
\item the ribbon edge between two vertices in $\widehat \Gamma$ is given a half twist if the sign between the corresponding edges of $\Gamma$ is negative, and no twist otherwise. 
\end{itemize}

In this way, the $n$ edges adjacent to any vertex in $\Gamma$ become a cycle of length $n$ in $\widehat\Gamma$ which is homeomorphic to a M\"obius band, because there is an odd number of half twists. Adjacent vertices in   $\Gamma$ correspond to M\"obius bands that intersect transversely in $\widehat\Gamma$.
 
To go from $\widehat{\Gamma}$ to $X(\Gamma)$, simply inflate each $4$-valent vertex to a cross $C(t_n)$. Associate the edges around the vertex to the four boundary components of $C(t_n)$ so that the cyclic order goes: left, bottom, right, top. Then glue crosses with or without half twist according to whether the edges of $\widehat{\Gamma}$ have a half twist or not.

From the surface $X(\Gamma)$, we can go back to $\widehat\Gamma$ by collapsing the front and back of each cross (i.e., taking the quotient of $X(\Gamma)$ by the reflection across the seams) then taking the graph dual to the decomposition of the resulting surface into octagons. Note that in this way, the seams of $X(\Gamma)$ correspond to the boundary components of $\widehat\Gamma$. 

\subsection{The $n=2$ case} \label{subsec:n=2}

A $2$-regular signed graph $\Gamma$ does not appear to carry e\-nough in\-for\-mation to prescribe how to glue crosses together. For instance, there is only one cyclic ordering on two elements, whereas there are two distinct directions of travel along a ring made with two crosses. 

For $n=2$, we start directly with a graph $G$ playing the role of $\widehat \Gamma$ instead. That is, let $G$ be a finite, connected, $4$-regular, ribbon graph such that any path in $G$ which does not turn (i.e., goes to the opposite edge in the cyclic order at each vertex) is closed of length $2$, and has a neighborhood homeorphic to a M\"obius band. Given such a graph $G$, we obtain a surface $X(G)$ by replacing each vertex of $G$ with a cross $C(t_2)$ and gluing them in the prescribed way as in the previous subsection. The resulting surface $X(G)$ is such that each of its crosses belongs to two rings isometric to $R(2,t_2)$.

We claim that there are two isomorphism classes of such graphs $G$ with $V$ vertices if $V\geq 2$ is a multiple of $3$, and only one isomorphism class otherwise. 

Pick any M\"obius band $B$ of length two in $G$ and cut $G$ along the two edges of $B$. The resulting object $H$ has two vertices that have two opposite half-edges not connected to anything.  Pick either of these vertices, start on one side of it, and start walking along an uncut edge. At the next encountered vertex, turn left, and so on, until you reach a dead end. In this way, the path traced is a boundary component of $H$ which passes through each vertex only once.

We can draw the ribbon graph $H$ in the plane as a tubular neighborhood of a regular $V$-gon with its sides extended a little bit, one side cut open, and the ends of each uncut side glued via a half twist (see \figref{fig:cutgon}). The left-turning path traced above corresponds to the inner boundary component of this cut $V$-gon. 

\begin{figure}[htp]
\centering
{\includegraphics{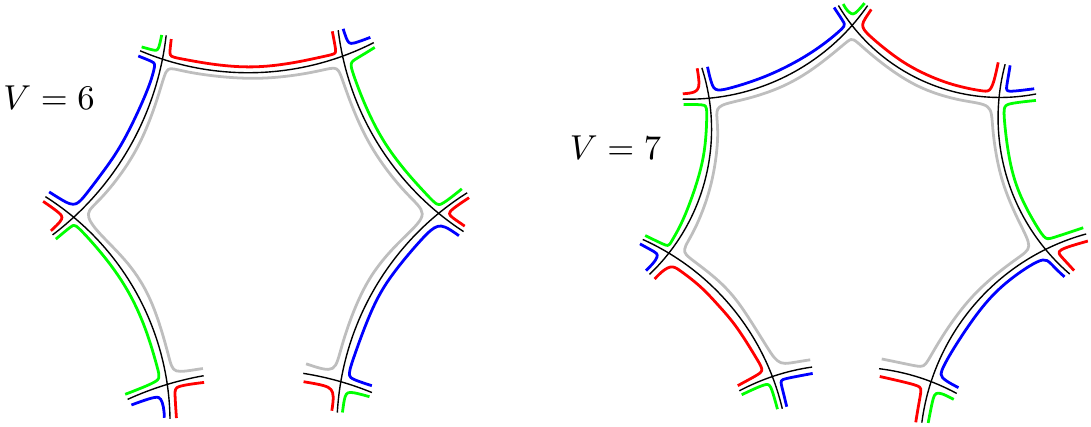}}
\caption{A representation of the ribbon graph $H$ with $6$ and $7$ vertices. The ends of each long segment are glued with a half twist. This leaves four half-edges that need to be paired up.}\label{fig:cutgon}
\end{figure}

The graph $G$ is obtained from $H$ by pairing up the two free half-edges of the first vertex with the two free half-edges of the last vertex, and giving one pair a half twist. There are two ways to pair them, and two choices for which pair gets a half twist, for a total of four choices (see \figref{fig:4choices}). However, some of these choices yield isomorphic objects. To see this, color the four boundary components of $H$ gray, red, green and blue. In the planar representation, $H$ has $2V+2$ ends and $2V+2$ gaps between these ends, one of which is on the inside. Each outer gap is connected (via half twists at the ends of extended sides) to the third next gap. This is why the residue of $V$ modulo $3$ is relevant.

\begin{figure}[htp]
\centering
{\includegraphics[width=\textwidth]{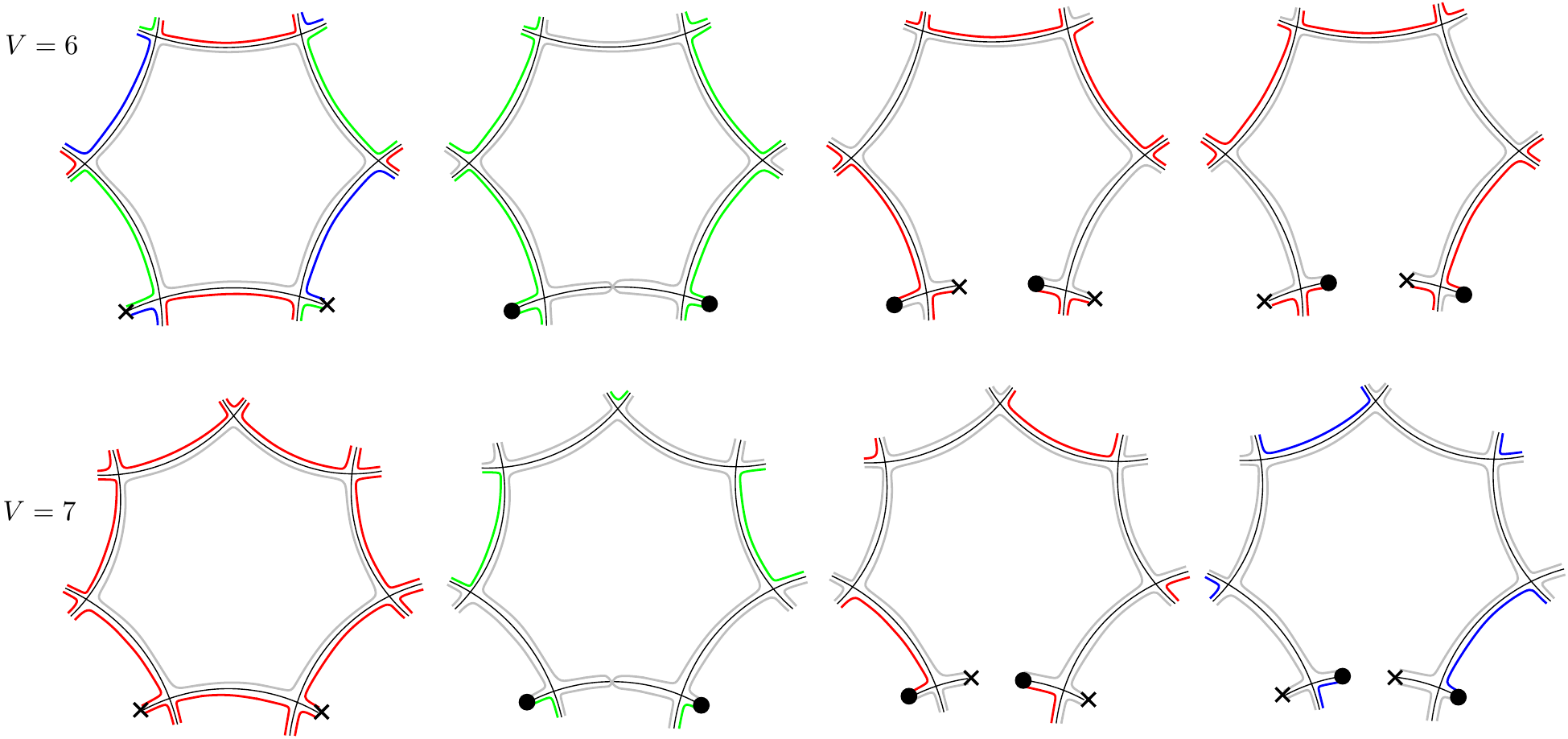}}
\caption{The four admissible pairings with $6$ and $7$ vertices. The crosses indicate a half twist whereas the dots indicate a lack thereof. Sides with the same color belong to the same boundary component.}\label{fig:4choices}
\end{figure}

To fix ideas, color the inner gap gray and the first outer gap, as well as those that it is connected to, in red. Similarly, color the other two boundary components green and blue. The last outer gap gets colored red if and only if $2V$ (hence $V$) is a multiple of $3$. Assume this is the case. Then at each of the first and last vertices there is one free half-edge with one side gray and one side red, and one free half-edge with one side green and one side blue.

As indicated earlier, there are four ways to close up $H$:
\begin{itemize}
\item If we glue gray to gray and red to red, then green gets glued to green and blue to blue. The resulting ribbon graph $G$ has four boundary components of length $V$ each. 
\item If we glue gray to red, then green gets glued to blue. The resulting ribbon graph $G$ has two boundary components of length $2V$ each.   
\item If we glue gray to blue, then red gets glued to green. The resulting ribbon graph $G$ has two boundary components of length $2V$ each. 
\item If we glue gray to green, then red gets glued to blue. The resulting ribbon graph $G$ has two boundary components of length $2V$ each.
\end{itemize}
The four possibilities are depicted on the first row of  \figref{fig:4choices} for $V=6$. One can check that the last three ribbon graphs are all isomorphic via cut-and-paste, so we indeed get two distinct isomorphism classes. 

Suppose that $V$ is not a mutiple of $3$. Then if two colors are on two sides of the same free half-edge of the first vertex in $H$, they are on different free half-edges of the last vertex and vice versa. In this case, it is not possible to glue each color to itself, nor is it possible to connect the colors in two pairs. Whichever of the four admissible gluings we pick, one color closes up while the three other colors connect together (see the second row of \figref{fig:4choices} for $V=7$). That is to say, any ribbon graph $G$ as above with $V \neq 0 \mod 3$ vertices has one boundary component of length $V$ and one boundary component of length $3V$. This implies that we can represent $G$ as a tubular neighborhood of a regular $V$-gon in the plane with sides extended and all half twists on the outside (as in the first column of \figref{fig:4choices}). In other words, there is only one isomorphism class.
   
\begin{remark}
In the sequel, we will not distinguish between the case $n=2$ and $n\geq 3$. That is, we will abuse notation and speak of the surface $X(\Gamma)$ for a $2$-regular signed graph $\Gamma$. In those instances, one should take $X(\Gamma)$ to be any of the surfaces $X(G)$ for graphs $G$ as above with the same number of vertices as $\Gamma$. 
\end{remark}

\subsection{Systoles}

We will show that the systoles in the surface $X(\Gamma)$ defined above are the $a$-, $b$- and $c$-curves, provided that $\Gamma$ has sufficiently large girth. The \emph{girth} of a graph is defined as the length of its shortest non-trivial loop. The problem with graphs with small girth is that the seams of the crosses in $X(\Gamma)$ can close up to form short geodesics. The following lemma shows that the seams are indeed the main thing to worry about.

\begin{lem} \label{lem:sigma_shortest}
Let $t>0$. The shortest non-trivial arcs in the cross $C(t)$ with endpoints in the boun\-dary are the seams, which have length $\sigma(t)$.
\end{lem}
\begin{proof}
The only potential candidates for shortest arcs are the seams or the other axes of symmetry of $C(t)$. Indeed, any arc that intersects one of the loci of reflection can be shortened by surgery with its reflection unless it coincides with the latter. Moreover, these loci cut the cross into topological disks. Since each seam is homotopic to a surgery on one horizontal and one vertical axis, the seams are shortest.
\end{proof}

We can now prove the main result of this section.

\begin{thm} \label{thm:systoles_closed_surface}
Let $n \geq 2$ and let $\Gamma$ be a connected, $n$-regular, signed graph of girth larger than $a(t_n) / \sigma(t_n)$. Then the systoles in the surface $X(\Gamma)$ are the $a$-, $b$- and $c$-curves, which have length $a(t_n)$. If $\Gamma$ is finite, then the genus $g$ of $X(\Gamma)$ is equal to $E + 1$ where $E$ is the number of edges in $\Gamma$ and there are $(12g-12)$ systoles in $X(\Gamma)$.
\end{thm}

\begin{remark}
The girth of a tree is infinite by convention, hence \thmref{thm:systoles_closed_surface} generalizes \propref{prop:systoles_in_tree}.
\end{remark}

\begin{proof}
Let $\gamma$ be a systole of $X(\Gamma)$. We define the shadow $s(\gamma)$ in the graph $\Gamma$ in the same way as in the proof of \propref{prop:systoles_in_tree}. If $s(\gamma)$ is non-contractible in $\Gamma$, then it traverses more than $a(t_n) / \sigma(t_n)$ edges by hypothesis. This means that $\gamma$ traverses as many crosses, hence is longer than \[(a(t_n) / \sigma(t_n)) \cdot \sigma(t_n) = a(t_n)\]
by \lemref{lem:sigma_shortest}, contradiction. It follows that $s(\gamma)$ is a contractible loop, so that it lifts to the universal cover of $\Gamma$, the $n$-regular tree $T(n)$. The tree of rings $\tree(n)$ similarly covers $X(\Gamma)$ and $\gamma$ lifts to $\tree(n)$. By \propref{prop:systoles_in_tree}, any lift $\wtilde{\gamma}$ is one of the curves of type $a$, $b$ or $c$ in a ring or a pair of transverse rings of $\tree(n)$. Since the covering map $\tree(n) \to X(\Gamma)$ is injective on each ring and each pair of transverse rings, $\gamma$ itself is an $a$-, $b$- or $c$-curve.

Let $g$ be the genus of $X(\Gamma)$. There are $4n$ curves of type $a$ or $b$ per ring, $n$ crosses per ring, and $2$ rings per cross, hence $8$ such curves per cross. Since each cross has Euler characteristic $-2$, there are $(g-1)$ crosses in $X(\Gamma)$, hence $(8g-8)$ curves of type $a$ or $b$ in total. Since each cross is central to exactly one pair of transverse rings and there are four $c$-curves per pair, the number of $c$-curves is equal to $(4g-4)$. By construction, the number of crosses is equal to the number $E$ of edges in $\Gamma$. Note that the number $V$ of vertices in $\Gamma$ satisfies $nV = 2E$ since $\Gamma$ is regular of degree $n$.
\end{proof}

\begin{remark}
The numbers $L_n$ and $w_n $ in \thmref{thmA} are defined as $a(t_n)$ and $a(t_n)/\sigma(t_n)$ respectively.
\end{remark}

As we will see in \subsecref{subsec:length}, the quantity $a(t_n)/\sigma(t_n)$ grows exponentially with $n$. Therefore the girth of $\Gamma$---and hence the genus of $X(\Gamma)$---has to be very large for the above result to hold. The first order of business, however, is to show that the surfaces obtained are local maxima of the systole function. This is proved in the next two sections.

\section{The lengths of the systoles determine the surface locally} \label{sec:regularity}

In this section, we show that the systoles in $X(\Gamma)$ can detect any infinitesimal movement, that is, the derivative of their length is injective on the tangent space to Teichm\"uller space.

\subsection{Twist deformations}

Given a simple closed geodesic $\beta$ in a hyperbolic surface $X$, we denote by $\tau_\beta$ the infinitesimal Fenchel-Nielsen twist deformation along $\beta$. More precisely, 
\[
\tau_\beta = \left.\frac{d}{dt}\right|_{t=0} X_t    
\]
where $X_t$ is the surface obtained by cutting $X$ along $\beta$, twisting distance $t$ to the left, then regluing. Given any closed geodesic $\alpha \subset X$, the \emph{cosine formula} says that
\begin{equation} \label{eqn:wolpert}
\frac{\partial \ell_\alpha}{\partial \tau_\beta} = \sum_{p \in \alpha \cap \beta} \cos \angle_p(\alpha, \beta)
\end{equation} 
where $\angle_p(\alpha, \beta)$ is the counter-clockwise angle from $\alpha$ to $\beta$ at the point $p$ \cite{WolpertTwist,Kerckhoff}.

For every $n\geq 1$, the systoles in the ring $R(n,t_n)$ include the curves $a_1, \ldots, a_{2n}$ and $b_1, \ldots, b_{2n}$ by \propref{prop:systoles_in_ring} and \propref{prop:systoles_bolza}. We want to compute the effect of twisting around any of these curves on the length of any of them. To this end, let $M$ be the $4n \times 4n$ matrix whose $(i,j)$-th entry is the derivative of the length of the $i$-th curve in the set $S=\{a_1, \ldots, a_{2n},b_1, \ldots, b_{2n}\}$ with respect to the twist deformation along the $(2n+j)$-th curve (modulo $4n$) in $S$. Recall that the $a$-curves are pairwise disjoint, as are the $b$-curves, and that each $a_i$ intersects each $b_j$ exactly once (see \figref{fig:curves}). The cosine formula \eqref{eqn:wolpert} thus gives
\[
M_{i,j} = \begin{cases}
\cos \angle (a_i, b_j) & \text{if }i,j \in \{1,\ldots,2n\}\\
\cos \angle (b_i, a_j) & \text{if }i,j \in \{2n+1,\ldots,4n\}\\
0 & \text{otherwise}.
\end{cases}
\]
In other words, $M$ is block diagonal of the form
\[
M = 
\begin{pmatrix} 
A & 0 \\
0 & -A^\intercal
\end{pmatrix}.
\] 
In particular, $M$ is invertible if and only if $A$ is. In the following two subsections we will show that $A$ (and hence $M$) is indeed invertible.

\begin{prop} \label{prop:full_rank}
For any $n \geq 1$, the matrix $M$ of derivatives of lengths of $a$- and $b$-curves in the ring $R(n,t_n)$ with respect to the twist deformations around these curves has full rank.
\end{prop}

An immediate consequence is that the twists deformations around the $a$- and $b$-curves form a basis of the tangent space to the Teichm\"uller space of the ring.

\begin{cor} \label{cor:twists_generate}
The twist deformations around the $a$- and $b$-curves in the ring $R(n,t_n)$ form a basis of the tangent space to the Teichm\"uller space of $R(n,t_n)$ with fixed boundary lengths for any $n \geq 1$.
\end{cor}
\begin{proof}
The ring $R$ is a surface of genus $1$ with $2n$ boundary components. As such, it admits a pants decomposition with $2n$ interior curves. The Fenchel--Nielsen coordinates for these interior curves parametrize the Teichm\"uller space with fixed boundary lengths. Hence the latter has dimension $4n$, as does its tangent space at the point $R$. By \propref{prop:full_rank}, the twist deformations about the $a$- and $b$-curves in $R$ are linearly independent. Since there are $4n$ such curves, these tangent vectors form a basis of the tangent space.
\end{proof}

In order to prove that the matrix $A$ of cosines of angles has full rank, we need to estimate these angles. It turns out that each column in $A$ has one entry close to $1$ and the other entries fairly close to $-1$. For instance, when $n=3$ one can compute that
\[
A \approx
\begin{pmatrix}
\phantom{-}0.961 & -0.652 & -0.924 & -0.962 & -0.924 & -0.652\\
-0.652 &  \phantom{-}0.961 & -0.652 & -0.924 & -0.962 & -0.924\\
-0.924 & -0.652 &  \phantom{-}0.961 & -0.652 & -0.924 & -0.962\\
-0.962 & -0.924 & -0.652 &  \phantom{-}0.961 & -0.652 & -0.924\\
-0.924 & -0.962 & -0.924 & -0.652 &  \phantom{-}0.961 & -0.652\\
-0.652 & -0.924 & -0.962 & -0.924 & -0.652 &  \phantom{-}0.961
\end{pmatrix}
\] 
From this pattern we will deduce that $0$ is not an eigenvalue of $A$.

\subsection{Angle estimate}

Let $\theta=\theta(n)$ be the angle from $e$ to any of the curves $a_j$ in the ring $R(n,t_n)$. Then the angle from any $b_j$ to $e$ is also equal to $\theta$. Also let $\phi_j$ be the counter-clockwise angle from $a_j$ to $b_1$. Recall that there are $2n$ curves $a_j$ that are images of one another by the shift $\eta : R(n,t_n) \to R(n,t_n)$ which translates by distance $e/2n$ along the curve $e$. In particular, the curves $a_j$ intersect $e$ at regularly spaced intervals of length $e/2n$ each. Therefore $b_1$, $e$ and $a_j$ together bound an isoceles triangle whose base has length $|n+1-j| \cdot e / 2n$, whose angles at the base are equal to $\theta$ and whose third angle is equal to $\phi_j$ (see \figref{fig:angles}). This holds for every $j \in \{1,\ldots, 2n\}$ except for $j=n+1$, where we get a triple intersection between $b_1$, $e$ and $a_{j+1}$.  

\begin{figure}[htp]
\centering
{\includegraphics{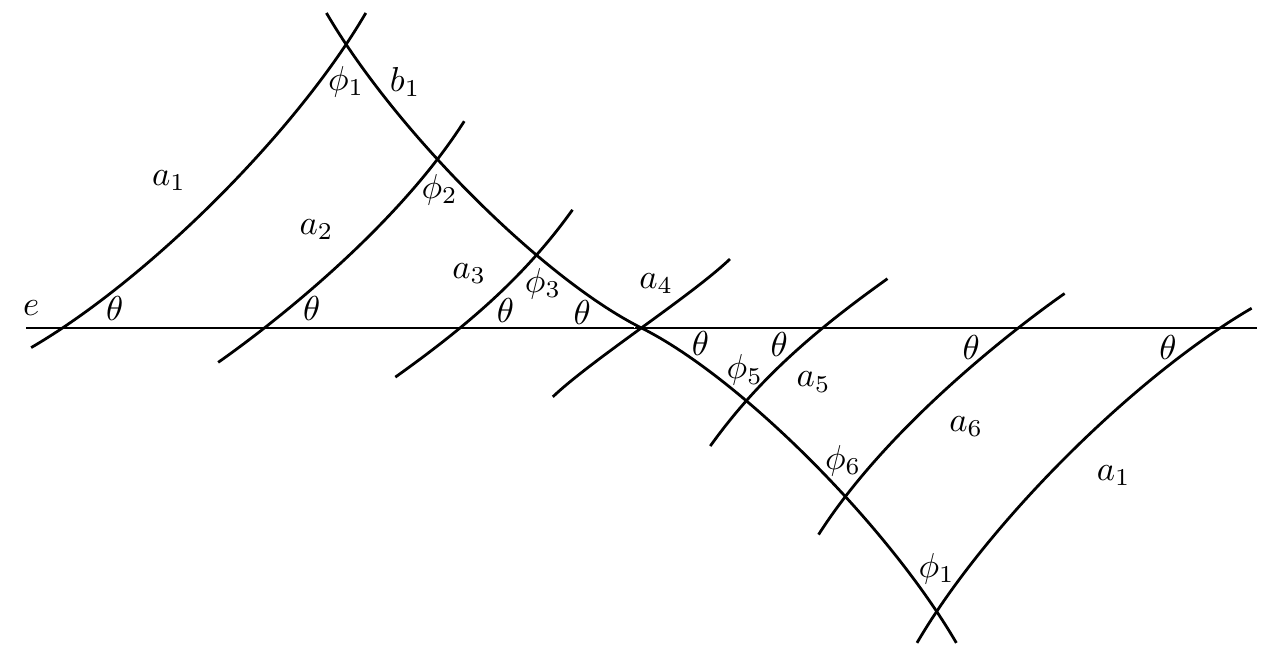}}
\caption{The isoceles triangles bounded by the curves $b_1$, $e$ and $a_j$ in the ring $R(n,t_n)$ for $n=3$}\label{fig:angles}
\end{figure}

By dropping the altitude in each isoceles triangle and applying \eqnref{eq:righttriangle2} we obtain
\[
\cos \frac{\phi_j}{2} = \sin\theta \cosh\left((n+1-j) \frac{e}{4n}\right). 
\]
The double angle formula for cosine yields
\begin{equation} \label{eqn:angles}
\cos \phi_j = 2\sin^2\theta \cosh^2\left((n+1-j) \frac{e}{4n}\right) - 1.
\end{equation}
Observe that the formula holds for $j=n+1$ as well since $\phi_{n+1}+2\theta = \pi$ so that
\[\cos \phi_{n+1} = \cos(\pi - 2 \theta) = -\cos(2\theta) = 2\sin^2 \theta - 1.\]

We will show that the first angle $\phi_1$ is very small whereas the following angles $\phi_j$ are close to $\pi$. We first need an elementary inequality involving sums of hyperbolic cosines.

\begin{lem} \label{lem:sumcosh}
For any $n \geq 1$ and any $x > \arccosh(\sqrt{2})$ we have
\[
2 \sum_{k=0}^{n-1} \cosh^2(kx) < \cosh^2(nx).
\]
\end{lem}
\begin{proof}
We proceed by induction on $n$. For $n=1$, the inequality reduces to $2 < \cosh^2(x)$ which is true by hypothesis. Now suppose that 
\[
2 \sum_{k=0}^{n-1} \cosh^2(kx) < \cosh^2(nx)
\]
for some $n\geq 1$. Then
\begin{equation} \label{eqn:induction}
2 \sum_{k=0}^{n} \cosh^2(kx) < 3\cosh^2(nx)
\end{equation}
which we want to show is less than $\cosh^2((n+1)x)$.

The addition formula for hyperbolic cosine gives
\begin{align*}
\cosh((n+1)x) &= \cosh(nx)\cosh(x)+ \sinh(nx)\sinh(x) \\
&> \sqrt{2}\, \cosh(nx) + \sinh(nx) \\
&= \left\{\sqrt{2} + \tanh(nx)\right\}\cosh(nx) \\
&> \left\{\sqrt{2} + \tanh(\arcsinh(1))\right\} \cosh(nx) \\
&= \left( \sqrt{2} + \frac{1}{\sqrt{2}} \right) \cosh(nx) \\
&> \sqrt{3}\, \cosh(nx)
\end{align*}
where we used the fact that $nx \geq x >\arccosh\left(\sqrt{2}\right)=\arcsinh(1)$. Putting this back in \eqnref{eqn:induction} gives
\[
2 \sum_{k=0}^{n} \cosh^2(kx) < 3\cosh^2(nx) < \cosh^2((n+1)x).
\]
By induction, the inequality holds for any $n\geq 1$.
\end{proof}

We can now show that the first angle $\phi_1$ is much closer to $0$ than any of the other angles, which approach $\pi$ rapidly as $j$ increases to $n+1$. The precise statement is expressed in terms of the cosines of the angles.

\begin{lem} \label{lem:angles}
For any $n \geq 1$, the angles $\phi_j$ from $a_j$ to $b_1$ satisfy 
\[ \sum_{j=2}^{2n} (\cos \phi_j +1) < (\cos \phi_1 + 1).  \]
\end{lem}
\begin{proof}
By \eqnref{eqn:angles} this inequality is equivalent to 
\begin{equation} \sum_{j=2}^{2n} \cosh^2\left((n+1-j) \frac{e}{4n}\right)  < \cosh^2(e/4). \label{eqn:reformulation}
\end{equation}
Each summand on the left appears twice except for $j=n+1$ so that
\[
\sum_{j=2}^{2n} \cosh^2\left((n+1-j) \frac{e}{4n}\right) < 2 \sum_{k=0}^{n-1} \cosh^2\left(k\frac{e}{4n}\right).
\]
Recall that $e=4n\arcsinh(\coth(t_n)) > 4n \arcsinh(1)$ and hence \[\frac{e}{4n} > \arcsinh(1) = \arccosh\big(\sqrt{2}\big). 
\]
We can therefore apply \lemref{lem:sumcosh} with $x = e/4n$ to obtain the desired inequality
\[
\sum_{j=2}^{2n} \cosh^2\left((n+1-j) \frac{e}{4n}\right) < 2 \sum_{k=0}^{n-1} \cosh^2\left(k\frac{e}{4n}\right) < \cosh^2(e/4).
\]
\end{proof}

\begin{cor} \label{cor:sum_angles}
For any $n \geq 1$, the angles $\phi_j$ from $a_j$ to $b_1$ satisfy 
\[ \sum_{j=1}^{2n} \cos \phi_j \neq 0.  \]
\end{cor}

\begin{proof}
First assume that $n\geq 2$. The above statement is equivalent to \[\sum_{j=1}^{2n} (\cos \phi_j+1) \neq 2n.\] By the previous lemma we have $\sum_{j=1}^{2n} (\cos \phi_j+1)< 2(\cos \phi_1 + 1) < 4 \leq 2n$.

If $n=1$, then $a_1$ meets $b_1$ at right angle since both of them intersect the $f$-curve with angle $\pi/4$. Furthermore, $\phi_2=3\pi /4$ (see \figref{fig:bolza}). Therefore \[\cos \phi_1 + \cos \phi_2 = -\sqrt{2}/2 \neq 0.\]
\end{proof}

\subsection{The Gershgorin circle theorem}

If the diagonal entries of a matrix dominate the rest, then the matrix is invertible. More generally, one can deduce information about the location of the eigenvalues from the size of the entries \cite{Gershgorin}. 

\begin{thm}[Gershgorin]
Let $U$ be an $n \times n$ matrix with entries $u_{i,j}$. Then the eigenvalues of $U$ are contained in the union of the closed disks with centers $u_{j,j}$ and radii $\sum_{i \neq j}|u_{i,j}|$. In particular, if $|u_{j,j}|>\sum_{i \neq j}|u_{i,j}|$ for every $j$, then $U$ is invertible.
\end{thm}

The last sentence of the theorem is quite transparent: if $x\in \R^n$ is non-zero and $x_j$ is its largest entry in absolute value, then $x$ times the $j$-th column of $U$ is non-zero by the triangle inequality.

We apply this criterion to the matrix $A + J$ where $A$ is the matrix of cosines of angles from the $a$-curves to the $b$-curves in the ring $R(n,t_n)$ and $J$ is the $2n \times 2n$ matrix whose entries are all equal to one.

\begin{lem}
The matrix $A+J$ is invertible for any $n\geq 1$.
\end{lem} 
\begin{proof}
Note that the entries of $A+J$ are positive so we do not need to take absolute values. \lemref{lem:angles} shows that the first entry of the first column of $A+J$ is larger than the sum of the other entries in that column. By symmetry of the ring, the entries of $A$ satisfy $A_{i+1,j+1} = A_{i,j}$ where the indices are taken modulo $2n$, and similarly for $A+J$. The Gershgorin circle theorem thus implies that $A+J$ is invertible.   
\end{proof}

It is easy to deduce that $A$ itself is invertible.

\begin{lem}
The matrix $A$ of cosines of angles from the $a$-curves to the $b$-curves in the ring $R(n,t_n)$ is invertible for any $n\geq 1$.
\end{lem} 
\begin{proof}
Let $V$ be the orthogonal complement of the vector $u=(1,\ldots,1)^\intercal$ in $\R^{2n}$. Since the restriction of $J$ to $V$ is equal to zero, $A$ and $A+J$ act the same way on $V$. Moreover, $A$ and $A+J$ both send the span of $u$ onto itself since they have constant non-zero row sums. The row sums are all the same because the rows are cyclic permutations of one another. The row sums of $A+J$ are non-zero because its entries are positive and the row sums of $A$ are non-zero by \corref{cor:sum_angles}. Since $A+J$ is surjective by the previous lemma, we obtain
\[
\R^{2n} = (A+J)(\R^{2n}) = (A+J)(V+\spann u) = A(V + \spann u) = A (\R^{2n}).
\]
We conclude that $A$ itself is surjective, hence invertible. 
\end{proof}

This implies that the full matrix $M$ of cosines of angles between all systoles in the ring $R(n,t_n)$ is invertible.

\begin{proof}[Proof of \propref{prop:full_rank}]
We have $\det(M)=\det(A)^2 \neq 0$ by the previous lemma.
\end{proof}

\subsection{From rings to closed surfaces}

Let $X=X(\Gamma)$ be a closed surface  obtained by gluing crosses $C(t_n)$ along a connected, finite, $n$-regular, signed graph $\Gamma$ of sufficiently large girth as in \thmref{thm:systoles_closed_surface}, so that its systoles are the $a$-, $b$- and $c$-curves (and $f$-curves if $n=1$). We now prove that the lengths of these curves determine the surface in some neighborhood of $X$.

\begin{thm} \label{thm:lengths_determine}
Let $n \geq 1$ and let $X=X(\Gamma)$ where $\Gamma$ is a finite, connected, $n$-regular, signed graph of girth larger than $a(t_n)/\sigma(t_n)$. The map $\T(X) \to \R_+^S$ sending $Y$ to the vector of lengths $(\ell_\gamma(Y))_{\gamma \in S}$ where $S$ is the set of systoles in $X$ has injective derivative at the point $X$.
\end{thm}

\begin{proof}
We need to show that the differentials $\{d\ell_{\gamma}\}_{\gamma \in S}$ span the cotangent space $T_X^* \T(X)$ over $\R$. Wolpert's twist-length duality \cite[Theorem 2.10]{WolpertDuality} states that
\[
d\ell_{\gamma} = \sqrt{-1} \left(\frac{\partial}{\partial \tau_\gamma} \right)^*
\]
for any simple closed geodesic $\gamma$, where the dual is taken with respect to the Weil--Petersson metric. Therefore, the  length differentials $\{d\ell_{\gamma}\}_{\gamma \in S}$ span the cotangent space $T_X^* \T(X)$ if and only if the twist deformations $\{\partial / \partial \tau_{\gamma}\}_{\gamma \in S}$ span the tangent space $T_X \T(X)$. 

By \propref{prop:systoles_bolza} (for $n=1$) and \thmref{thm:systoles_closed_surface} (for $n \geq 2$), the set $S$ of systoles includes the $a$- and $b$-curves. We will show that the twist deformations about the $a$- and $b$-curves generate the tangent space. To see this, observe that there exists a pants decomposition $\calP$ of $X$ consisting entirely of curves that are each in the interior of some ring $R \subset X$. For instance, one can take $\calP$ to be the set of all $f$-curves in $X$ (the curves that cut $X$ into crosses) together with one curve in each cross that separates it into two pairs of pants---call these $d$-curves. Each $d$-curve is in the interior of both rings that it intersects, while each $f$-curve is in the interior of a unique ring. 

The lengths and twists around the curves in the pants decomposition $\calP$ define Fenchel--Nielsen coordinates $\T(X) \to (\R_+ \times \R)^\calP$ once a convention  is chosen for what zero twist means. For any curve $\alpha \in \calP$, let $R\subset X$ be a ring that contains $\alpha$ in its interior. By \corref{cor:twists_generate}, the twist deformations about the $a$- and $b$-curves in $R$ generate the tangent space to the Teichm\"uller space of $R$ with fixed boundary lengths. In particular, the two tangent vectors corresponding to changing the length or twist parameter of $\alpha$ at unit speed while keeping all the other Fenchel--Nielsen coordinates fixed are in the span of the twist deformations around the 
$a
$- and $b$-curves in $X$. Since the Fenchel--Nielsen length and twist parameters define a smooth coordinate system for $\T(X)$, we are done.     
\end{proof}

\begin{remark}
The proof actually shows that the derivative of the vector of lengths of all the $a$-curves and $b$-curves is injective at $X(\Gamma)$. The $c$-curves are not needed for this; they only play a role in the next section.  
\end{remark}

\begin{remark}
In \cite{SchmutzGerman}, Schmutz Schaller describes a collection of $(6g-5)$ curves such that their lengths define a topological embedding of Teichm\"uller space into $\R_+^{6g-5}$. See also \cite{HamParam1,HamParam2}. If $g$ is the genus of $X(\Gamma)$, then there are $(8g-8)$ curves of type $a$ or $b$ in $X(\Gamma)$. We do not know if their lengths define a global embedding of Teichm\"uller space, but \thmref{thm:lengths_determine} in conjunction with the inverse function theorem implies that they define an embedding in a neighborhood of $X(\Gamma)$.  
\end{remark}

\section{The systole decreases under all deformations} \label{sec:minimality}

Let $X=X(\Gamma)$ where $\Gamma$ is a signed graph satisfying the hypotheses of \thmref{thm:lengths_determine}. Now that we know that the systoles in $X$ can detect any infinitesimal movement, it remains to show that at least one of them shrinks under any infinitesimal deformation. Even though we have proved that the twist deformations around the $a$-curves and $b$-curves in $X$ generate the tangent space $T_X \T(X)$, it will be convenient to use a different basis to show this.

To define this other basis, we first explain how it acts on individual crosses. Let $C$ be a cross with four boundary lengths equal to $4t_n$ as in \secref{sec:cross}. For each boundary component $\beta \subset C$ and $s > 0$, we define the deformed cross $C^\beta_s$ to be the four-holed sphere with $\beta$-boundary of length $4(t_n+s)$, the three other boundaries of length $t_n$, and with the same symmetries fixing $\beta$ that $C$ has, that is, a $\Z_2 \times \Z_2$ group generated by a front-to-back reflection and a left-and-right or top-to-bottom reflection depending on whether $\beta$ is medial or lateral respectively. 

For example, if $\beta$ is the left boundary component of $C$, then $C^\beta_s$ is obtained by taking a right-angled hexagon with left side $t_n+s$, top side of length $2t_n$ and right side of length $t_n$, then reflecting this hexagon across its bottom side to obtain a right-angled octagon, then doubling this octagon across the four sides with unspecified lengths (see \figref{fig:deform}).

\begin{figure}[htp]
\centering
{\includegraphics{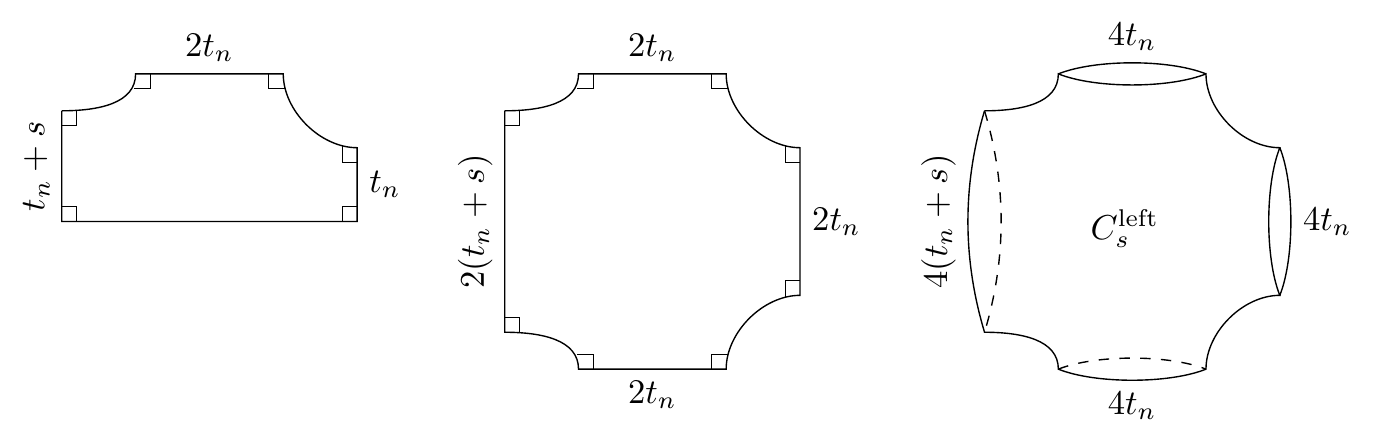}}
\caption{The length deformation of the cross about its left boundary}\label{fig:deform}
\end{figure}

Now if $\beta \subset X$ is any $f$-curve and $s>0$, then we define $X^\beta_s$ to be the same as $X$ but with the two crosses $C$ and $D$ adjacent to $\beta$ replaced with $C^\beta_s$ and $D^\beta_s$. These are glued together and with the other crosses in the most obvious way, without twisting. Finally, we let
\[
\lambda_\beta = \left.\frac{d}{ds}\right|_{s=0} X^\beta_s
\]
and call this the \emph{symmetric length deformation about $\beta$}. 

The motivation behind this construction is that the only canonical way to change the length of a curve on a surface is to flow along its gradient with respect to the Weil--Petersson metric. However, this gradient deformation is non-local in nature and its effect on the lengths of other curves (especially disjoint ones) is complicated to compute, although an explicit formula analogous to the cosine formula \eqref{eqn:wolpert} exists \cite[Equation (7)]{Riera}. 

The advantage of our symmetric length deformations is that the sum 
\[ \Lambda = \sum_{\beta \in \{ f\text{-curves}\}} \lambda_\beta\]
corresponds to expanding all the boundaries of all the crosses in $X$ at the same rate without twisting and while preserving the symmetries of all the crosses. In other words, the effect of $\Lambda$ is the same as increasing the parameter $t$ at unit speed in the definition of the ring $R(n,t_n)$. In particular, for any $a$- or $b$-curve $\alpha$ and for any $c$-curve $\gamma$ in $X$ we have
\begin{equation} \label{eqn:sign}
\frac{\partial \ell_\alpha}{\partial \Lambda} = a'(t_n)>0 \quad \text{and} \quad \frac{\partial \ell_\gamma}{\partial \Lambda} = c'(t_n) <0
\end{equation}
according to \lemref{lem:a_increases} and \lemref{lem:c_decreases} respectively.

To complete our basis for the tangent space $T_X \T(X)$ we also include the twist deformations around the $f$-curves as well as two more twist deformations $\tau_d$ and $\tau_h$ per cross. In the cross $C$, we pick the curve $d$ to be one of the two diagonal axes of symmetry and $h$ to be the curve depicted in \figref{fig:curves_in_cross}.

\begin{figure}[htp]
\centering
{\includegraphics{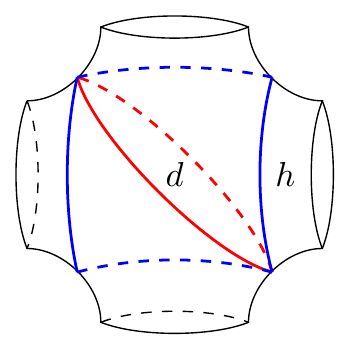}}
\caption{The curve $d$ (in red) and the curve $h$ (in blue) in the cross $C$}\label{fig:curves_in_cross}
\end{figure}

Observe that $d$ and $h$ intersect twice. Moreover, the oriented angles from $d$ to $h$ at the two intersection points are equal to each other since the rotation of angle $\pi$ about the centers of the front and back of the cross leaves each curve invariant, preserves orientation, and exchanges the two intersection points. Finally, the angle of intersection $\psi = \angle(d,h)$ is different from $\pi/2$ since the intersections occur at the midpoints of two opposite seams, and the seams are orthogonal to $d$ at those points. 

Let $L,R,T,B$ be the left, right, top and bottom boundaries of the cross respectively. The matrix of partial derivatives of the lengths of $\{L,R,T,B,d,h\}$ with respect to the deformations $\{\lambda_L,\lambda_R,\lambda_T,\lambda_B,\tau_h,\tau_d\}$ of the cross $C$ has the form
\[
\begin{blockarray}{ccccccc}
 & \lambda_L & \lambda_R & \lambda_T & \lambda_B & \tau_h & \tau_d \\
\begin{block}{c(cccccc)}               
d\ell_L & 1 & 0 & 0 & 0 & 0 & 0\\
d\ell_R & 0 & 1 & 0 & 0 & 0 & 0\\
d\ell_T & 0 & 0 & 1 & 0 & 0 & 0\\
d\ell_B & 0 & 0 & 0 & 1 & 0 & 0\\
d\ell_d & \delta & \delta & \delta & \delta &  2\cos \psi & 0\\
d\ell_h & \eps & \eps & \eps & \eps & 0 &  -2\cos \psi \\
\end{block}
\end{blockarray}
\]
for some $\delta, \eps \in \R$. It is lower triangular with non-zero diagonal entries, hence invertible. In particular, the deformations $\{\lambda_L,\lambda_R,\lambda_T,\lambda_B,\tau_h,\tau_d\}$ form a basis of the tangent space to the Teichm\"uller space of $C$ with variable boundary lengths.

\begin{lem}
The symmetric length deformations $\{ \lambda_\beta \}_{\beta \in \{ f\text{-curves}\}}$ together with the twist deformations $\{ \tau_\beta \}_{\beta \in \{ f\text{-curves}\}}$ and $\{ \tau_{d(C)} , \tau_{h(C)} \}_{C \in \{\text{crosses}\}}$ form a basis of $T_X\T(X)$.
\end{lem}  
\begin{proof}
It suffices to prove that these deformations span the tangent space since they are equal in number to its dimension. By the paragraph preceding the statement of this lemma, these deformations generate all the deformations of any cross in $X$. In particular, they generate the Fenchel--Nielsen length and twist deformations with respect to the pants decomposition of $X$ by $f$-curves and $d$-curves.
\end{proof}

We can now prove that the systole function decreases under all non-trivial deformations of $X$, hence that $X$ is a local maximum of $\sys$. This is a restatement of \thmref{thmA} from the introduction.

\begin{thm} \label{thm:local_max}
Let $n \geq 1$ and let $X=X(\Gamma)$ where $\Gamma$ is a finite, connected, $n$-regular, signed graph of girth larger than $a(t_n)/\sigma(t_n)$. Then for every non-zero tangent vector $v \in T_X \T(X)$, there is at least one systole $\alpha$ of $X$ such that $d \ell_\alpha (v) < 0$. That is, $X$ is a local maximum of the systole function.
\end{thm}
\begin{proof}
First assume that $n\geq 2$. Let $S$ be the set of systoles of $X$ and suppose that $v \in T_X \T(X)$ is such that $d \ell_\alpha (v) \geq 0$ for every $\alpha \in S$. We will show that $v=0$.

By the previous lemma, we can write
\[
v = \sum_{\beta \in F} \kappa_\beta \cdot \lambda_\beta + \sum_{\gamma \in D} \mu_\gamma \cdot \tau_\gamma
\]
for some $\kappa_\beta,\mu_\gamma \in \R$ where $F$ is the set of $f$-curves and $D$ is the set of all $f$-, $d$- and $h$-curves in $X$.

For every $\alpha \in S$ we have 
\[
0 \leq d\ell_\alpha(v) = \sum_{\beta \in F} \kappa_\beta \cdot d\ell_\alpha(\lambda_\beta) + \sum_{\gamma \in D} \mu_\gamma \cdot d\ell_\alpha(\tau_\gamma).
\]
Summing over any subset $Q \subset S$ we obtain
\[
0 \leq \sum_{\alpha \in Q} d\ell_\alpha(v) = \sum_{\beta \in F} \kappa_\beta \sum_{\alpha \in Q} d\ell_\alpha(\lambda_\beta) + \sum_{\gamma \in D} \mu_\gamma \sum_{\alpha \in Q} d\ell_\alpha(\tau_\gamma).
\]

Let $A$ be the set $a$- and $b$-curves in $X$ and let $C$ be the set of $c$-curves in $X$. The \textbf{first observation} is that if $Q$ is equal to either $A$ or $C$, then $\sum_{\alpha \in Q} d\ell_\alpha(\tau_\gamma)= 0$ for every $\gamma \in D$. Indeed, for every systole $\alpha \in Q$ intersecting $\gamma$, there is some systole $\alpha^* \in Q$ intersecting $\gamma$ with the supplementary angle. To see this, observe that in the tree of rings $\tree(n)$ there is an orientation-reversing isometry which sends (any lift of) $\gamma$ to itself and permutes the (lifts of) $a$- and $b$-curves, as well as the (lifts of) $c$-curves separately. For instance, if $\gamma$ is an $f$- or $d$-curve then the reflection of $\tree(n)$ across the seams works, and if $\gamma$ is an $h$-curve then the left-to-right reflection of the cross containing $\gamma$ extends to an isometry of $\tree(n)$. These reflections exchange $a$- and $b$-curves and preserve the set of $c$-curves. Hence the statement about angles coming in supplementary pairs holds in the tree of rings $\Sigma(n)$. Since $\tree(n)$ covers $X$ and every systole in $X$ is the image of a systole in $\tree(n)$, the statement holds in $X$ as well. The cosine formula \eqref{eqn:wolpert} thus implies that the total length variation of the curves in $Q$ is nil in the direction of $\tau_\gamma$. Thus
\begin{equation} \label{eqn:no_twist_term}
\sum_{\alpha \in A} d\ell_\alpha(v) = \sum_{\beta \in F} \kappa_\beta \sum_{\alpha \in A} d\ell_\alpha(\lambda_\beta) \quad \text{and} \quad \sum_{\alpha \in C} d\ell_\alpha(v) = \sum_{\beta \in F} \kappa_\beta \sum_{\alpha \in C} d\ell_\alpha(\lambda_\beta).
\end{equation}

The \textbf{second observation} is that the term
\[
\sum_{\alpha \in Q} d\ell_\alpha(\lambda_\beta)
\]
is independent of $\beta \in F$ when $Q$ is equal to either $A$ or $C$. Indeed, the deformation $\lambda_\beta$ only affects the lengths of systoles in pairs of transverse rings containing one of the two crosses adjacent to $\beta$. The geometry of the subsurface $Y \subset X$ containing all these pairs of transverse rings does not depend on $\beta$. This is because $Y$ is the union of the crosses corresponding to the edges of a subgraph $H \subset \Gamma$, namely, the $2$-neighborhood of a pair of consecutive edges in the cyclic order around a vertex (corresponding to the crosses meeting along $\beta$). Since $\Gamma$ is assumed to have girth larger than $a(t_n)/\sigma(t_n)$, and that number is bigger than $6$ (see Table \ref{table}), $H$ is a tree isometric to the $2$-neighborhood of \emph{any} pair of consecutive edges in $\Gamma$ (or in the $n$-regular tree). The resulting subsurface $Y$ and the total effect of the deformation $\lambda_\beta$ on the length of its $a$- and $b$-curves or its $c$-curves is therefore independent of $\beta$.

The \textbf{third and last observation} is that
\[
\sum_{\beta \in F} \left(\sum_{\alpha \in A} d\ell_\alpha(\lambda_\beta) \right)= \sum_{\alpha \in A} \left(\sum_{\beta \in F} d\ell_\alpha(\lambda_\beta)\right) = \sum_{\alpha \in A} d\ell_\alpha(\Lambda) = \sum_{\alpha \in A} a'(t_n) > 0
\]
by \eqnref{eqn:sign}, where $\Lambda = \sum_{\beta \in F} \lambda_\beta$. We deduce that $\sum_{\alpha \in A} d\ell_\alpha(\lambda_\beta)>0$ for any $\beta \in F$ by the second observation. Similarly,
\[
\sum_{\beta \in F} \left(\sum_{\alpha \in C} d\ell_\alpha(\lambda_\beta) \right)= \sum_{\alpha \in C} \left(\sum_{\beta \in F} d\ell_\alpha(\lambda_\beta)\right) = \sum_{\alpha \in C} d\ell_\alpha(\Lambda) = \sum_{\alpha \in C} c'(t_n) < 0
\]
so that $\sum_{\alpha \in C} d\ell_\alpha(\lambda_\beta) < 0$ for any $\beta \in F$.

For any fixed $\gamma \in F$ we have both
\[
0 \leq \sum_{\alpha \in A} d\ell_\alpha(v) = \left(\sum_{\beta \in F} \kappa_\beta\right) \left(\sum_{\alpha \in A} d\ell_\alpha(\lambda_\gamma) \right)
\]
and
\[
0 \leq \sum_{\alpha \in C} d\ell_\alpha(v) = \left(\sum_{\beta \in F} \kappa_\beta\right) \left(\sum_{\alpha \in C} d\ell_\alpha(\lambda_\gamma) \right)
\]
from \eqnref{eqn:no_twist_term} and the second observation. By the third observation, the sum in the rightmost parentheses is first positive then negative. We conclude that $\sum_{\beta \in F} \kappa_\beta = 0$ so that
$$
\sum_{\alpha\in A} d\ell_\alpha(v) = 0 = \sum_{\alpha\in C} d\ell_\alpha(v).
$$
Since each summand was assumed to be non-negative, they are all zero. By \thmref{thm:lengths_determine}, this implies that $v=0$.

If $n=1$, the same argument works with $A$ replaced by $F$. The point is that the $f$-curves are systoles in this case and their length increases under the deformation $\Lambda$.

It is easy to see that the first part of the theorem implies that $X$ is a local maximum of the systole function. For any unit vector $v \in T_X \T(X)$, the above implies that there is a curve $\alpha \in S$ and an $\eps_v > 0$ such that \[\sys(\exp_X(t v)) \leq \ell_\alpha(\exp_X(t v)) < \ell_\alpha(X) = \sys(X)\] for every $t \in (0,\eps_v)$, where $\exp_X(t v)$ is the point at distance $t$ from $X$ along the Weil--Petersson geodesic in the direction of $v$. Since the length functions $(\ell_\alpha)_{\alpha \in S}$ are conti\-nuous\-ly differentiable, $\eps_v$ can be chosen locally uniformly with respect to $v$. As the unit sphere in $T_X \T(X)$ is compact, there is an $\eps>0$ which works for all $v$. Hence there is a neighborhood $U$ of $X$ in $\T(X)$ such that $\sys(Y) \leq \sys(X)$ for every $Y \in U$ with equality only if $Y = X$. The same holds in moduli space.
\end{proof}

\section{Isometries are induced by graph isomorphisms} \label{sec:isometries}

In this section, we show that distinct signed graphs $\Gamma$ give rise to distinct hyperbolic surfaces $X(\Gamma)$. As a byproduct, we get that if the underlying graph has no non-trivial automorphism, then the resulting surface has a trivial automorphism group.

We first need to distinguish between the different kinds of systoles in $X(\Gamma)$.

\begin{lem} \label{lem:a_diff_from_c}
Let $n \geq 2$ and let $\Gamma$ be an $n$-regular signed graph such that the systoles in $X(\Gamma)$ are the $a$-, $b$- and $c$-curves. Then the $a$- and $b$-curves in $X(\Gamma)$ intersect a different number of systoles than the $c$-curves. 
\end{lem}
\begin{proof}

Let us count the number of systoles that intersect a given $c$-curve $\gamma$. In the pair of transverse rings $R_1 \cup R_2$ containing $\gamma$, there is a central cross at the intersection of the two rings and $2(n-1)$ non-central crosses. For each of the latter kind, $\gamma$ intersects only one side (front or back) of the cross, separating two opposite sides of that octagon. Thus for each non-central cross $C$, in the ring through $C$ distinct from $R_1$ and $R_2$, exactly half of the $a$- and $b$-curves intersect $\gamma$. This is because each $a$- and $b$-curve intersects only one side of each cross, connecting two opposite sides of that octagon. These curves contribute $2(n-1)\cdot 2n$ intersections. 

As for the systoles in $R_1$ or $R_2$, again half of them intersect $\gamma$. To see this, observe that any $a$- or $b$-curve is homotopic to a union of two geodesic segments: one that travels halfway along an $e$-curve and one that travel halfway along an $f$-curve. The $e$-curves are disjoint from $\gamma$ while each $f$-curve in $R_j$ intersects $\gamma$ once. Thus each $f$-curve in $R_j$ contributes one $a$-curve and one $b$-curve intersecting $\gamma$. This yields a total of $2n+2n=4n$ curves of type $a$ or $b$ in $R_1 \cup R_2$ that intersect $\gamma$.

How many $c$-curves intersect $\gamma$? We can first homotope any $c$-curve (including $\gamma$) to a union of two segments of $e$-curves. Each non-central cross in $R_1 \cup R_2$ is associated with four $c$-curves, half of which intersect $\gamma$. Indeed, when they are represented along the $e$-curves, any such $c$-curve $\zeta$ shares a segment $I$ with $\gamma$. At the extremities of $I$, the two curves $\gamma$ and $\zeta$ can turn toward either the same of different sides of $I$. In the first case we can homotope them to intersect only once while in the second we can homotope them to be disjoint. Thus there are $2 \cdot 2(n-1)$ curves of type $c$ that intersect $\gamma$ coming from the $2(n-1)$ non-central crosses in $R_1 \cup R_2$. 

Lastly, for each cross $C$ contained in a ring that intersects $R_1 \cup R_2$ such that $C$ is not itself contained in $R_1 \cup R_2$, we get two $c$-curves intersecting $\gamma$. There are $2(n-1)^2$ such crosses, accounting for $4(n-1)^2$ intersections.

Any other systole is disjoint from $\gamma$, being disjoint from $R_1 \cup R_2$. The total number of systoles intersecting $\gamma$ is thus
\[
4n(n-1)+4n+4(n-1)+4(n-1)^2 = 8n^2-4n.
\]

We now count the number of systoles that intersect a given $a$- or $b$-curve $\alpha$. In the ring $R$ where $\alpha$ lives, there are $2n$ systoles that intersect $\alpha$ apart from itself. 

Since $\alpha$ intersects each cross of $R$ in only one side (front or back) and separates two opposite sides of that octagon, it intersects exactly half of the $a$- and $b$-curves in each ring transverse to $R$. There are $n$ such rings, each contributing $2n$ intersections with $\alpha$.

These are all the $a$- and $b$-curves that $\alpha$ intersects. Now for the $c$-curves. By the above, $\alpha$ intersects half of the $c$-curves that intersect $R$. There are $n$ crosses per ring transverse to $R$, $n$ such rings, each contributing two $c$-curves that intersects $\alpha$, for a total of $2n^2$.

The number of systoles intersecting $\alpha$ is equal to
\[
2n+2n^2+2n^2 = 4n^2 +2n
\]
which is distinct from $8n^2 -4n$ for any $n\geq 2$ (the two real solutions are $0$ and $3/2$).
\end{proof}

The next step is to pick out pairs of $a$- and $b$-curves that are symmetric about the seams. In the notation of \subsecref{subsec:the_ring}, these are pairs $a_j$ and $b_j$ for some $j \in \{1, \ldots, 2n\}$. See \figref{fig:curves}.

\begin{lem} \label{lem:special_pair}
Let $n \geq 2$ and let $\Gamma$ be an $n$-regular signed graph. Then a pair of intersecting $a$- and $b$-curves in $X(\Gamma)$ maximizes the number of intersections with other $a$- and $b$-curves if and only if it is symmetric about the seams.
\end{lem}

\begin{proof}
Consider a pair $\alpha \cup \beta$ of $a$- and $b$-curves such that $\beta=\rho_\text{seams}(\alpha)$. What is special about this pair is that for each cross it intersects, it intersects both of its sides (front and back). Let $R$ be the ring containing $\alpha \cup \beta$. All the $4n$ systoles in $R$ intersect $\alpha \cup \beta$. Furthermore, all the $4n$ systoles in each of the $n$ rings transverse to $R$ intersect the pair $\alpha \cup \beta$. The total number of intersections is $4n^2+4n$.

Now suppose that $\alpha$ and $\beta$ are $a$- and $b$-curves contained in a common ring $R$ but are not symmetric about the seams. Then there is some cross $C \subset R$ such that $\alpha \cup \beta$ intersects only one side of $C$. Hence in the ring transverse to $R$ through $C$, only $2n$ systoles intersect $\alpha \cup \beta$, for a total of at most $4n^2+2n$ curves of type $a$ or $b$.

Finally, suppose that $\alpha$ and $\beta$ are not contained in a common ring. Let $R_1 \cup R_2$ be the pair of transverse rings containing them. In each ring transverse to $R_j$, there are $2n$ systoles that intersect $\alpha \cup \beta$ apart from $\alpha$ and $\beta$. Thus the number of $a$- and $b$-curves intersecting $\alpha \cup \beta$ is $2n \cdot n \cdot 2 + 2 = 4n^2 + 2$, which is less than $4n^2+4n$.
\end{proof}

We now have the required tools to prove that the map $\Gamma \mapsto X(\Gamma)$ is injective. We refer the reader back to \subsecref{subsec:signed_graph} for the definition of signed graphs and their isomorphisms, and to \subsecref{subsec:graph_to_surface} for the description of the map $\Gamma \mapsto X(\Gamma)$. 

\begin{thm} \label{thm:auto}
Let $n \geq 3$ and let $\Gamma_1$ and $\Gamma_2$ be $n$-regular signed graphs of girth larger than $a(t_n)/\sigma(t_n)$.  Any orientation-preserving isometry $X(\Gamma_1) \to X(\Gamma_2)$ is induced by a unique isomorphism of signed graphs $\Gamma_1 \to \Gamma_2$. 
\end{thm}
\begin{proof}
Let $\psi: X(\Gamma_1) \to X(\Gamma_2)$ be an orientation-preserving isometry. Then $\psi$ sends systoles of $X(\Gamma_1)$ to systoles of $X(\Gamma_2)$. By \lemref{lem:a_diff_from_c}, it sends the set of $a$- and $b$-curves on $X(\Gamma_1)$ to the set of $a$- and $b$-curves on $X(\Gamma_2)$. Furthermore, each pair of $a$- and $b$-curves in $X(\Gamma_1)$ that are symmetric about the seams in sent to a such a pair in $X(\Gamma_2)$ by \lemref{lem:special_pair}. 

The two angle bisectors of a symmetric pair of $a$- and $b$-curves at the intersection are along an $f$-curve and the seams. We may assume that the girth of $\Gamma_1$ and $\Gamma_2$ is larger than $2$ (see \subsecref{subsec:length}) so that the $f$-curves are distinguished from the seams. We conclude that $\psi$ sends $f$-curves to $f$-curves and seams to seams. In particular, it respects the decomposition of $X(\Gamma_1)$ and $X(\Gamma_2)$ into crosses. 

Let $E(\Gamma_j)$ be the set of edges of $\Gamma_j$. Since there is a bijection between the crosses in $X(\Gamma_j)$ and the edges in $\Gamma_j$, the isometry $\psi$ induces a bijection $\phi: E(\Gamma_1) \to E(\Gamma_2)$.  Since $\psi$ maps adjacent crosses to adjacent crosses, the induced map $\phi$ either preserves or reverses the cyclic order around each vertex. After applying a set of vertex flips to $\Gamma_2$ (which does not affect $X(\Gamma_2)$), we may assume that $\phi$ preserves cyclic orders. If two parallel rings in $X(\Gamma_1)$ have matching (resp. opposite) orderings, it is clear that $\psi$ sends them to parallel rings with matching (resp. opposite) orderings. That is, the sign between any two consecutive edges $p$ and $q$ in $\Gamma_1$ is the same as the sign between $\phi(p)$ and $\phi(q)$ in $\Gamma_2$. In other words, $\Gamma_1 = \Gamma_2$ up to isomorphism. 
\end{proof}

\begin{remark} \label{rem:n=2}
This statement is false for $n=1$ for the simple reason that there is no distinction between the $f$-curves and the seams (the $c$-curves). The analogous statement for $n=2$ is true (and the proof essentially identical) provided that we replace the signed graphs by $4$-regular ribbon graphs satisfying the conditions of \subsecref{subsec:n=2}. By the argument in that subsection, there are two isomorphism classes of such graphs whenever the number $V$ of vertices is a multiple of $3$, and one isomorphism class otherwise. Therefore, we get two distinct corresponding points in $\M_g$ if $g=V+1$ is congruent to $1 \mod 3$, and only one otherwise. 
\end{remark}

\begin{cor}
Let $n \geq 3$ and let $\Gamma$ be an $n$-regular signed graph of girth larger than $a(t_n)/\sigma(t_n)$. If $\Gamma$ has a trivial automorphism group, then so does $X(\Gamma)$.
\end{cor}

\begin{remark}
By an \emph{automorphism} of a hyperbolic surface, we mean an orientation-preserving isometry. Each surface $X(\Gamma)$ has at least one orientation-reversing isometry, namely the reflection across the seams.
\end{remark}

\section{Counting the number of examples in each genus} \label{sec:counting}

\subsection{Length estimates} \label{subsec:length}

In this subsection, we quantify how large the girth of the signed graph $\Gamma$ needs to be in terms of $n$ for the hypothesis of \thmref{thm:systoles_closed_surface} to be satisfied, that is, we estimate $a(t_n)/\sigma(t_n)$. In particular, we estimate the length $a(t_n)$ of the systoles of the resulting surface $X(\Gamma)$.

\begin{lem}  \label{lem:estimate1}
We have 
\[t_n = n\log\left(1+\sqrt{2}\right) + o(1) \quad \text{and} \quad a(t_n)=4n\log\left(1+\sqrt{2}\right)-2\log 2 +o(1)\]
as $n \to \infty$.
\end{lem}
\begin{proof}
Recall that the equality $a(t)=c(t)$ is equivalent to 
\begin{equation} \label{eqn:aequalsc}
\tanh(e(t)/4)\sinh(e(t)/4) =\cosh(t)
\end{equation}
by the proof of \lemref{lem:existenceanduniqueness}, 
and that $e(t)/4 = n \arcsinh(\coth(t))$. 

Let $\eps>0$ and let $\lambda = \log\left(1 + \sqrt{2}\right) = \arcsinh(1)$. We will show that if $n$ is large enough then the difference between the left-hand side (LHS) and the right-hand side (RHS) of \eqnref{eqn:aequalsc} switches sign when $t$ is between $n\lambda - \eps$ and $n\lambda + \eps$.

First observe that $e(t)/4 > n \arcsinh(1) = n \lambda$ for every $t>0$ and every $n$. Moreover, if $t\geq n\lambda - \eps$, then $e(t)/4 \leq e(n\lambda - \eps)/4=n \lambda + o(1)$ as $n \to \infty$. Now 
\[\tanh(x)\sinh(x)= \exp(x)/2+o(1) \quad  \text{and} \quad \cosh(x)=\exp(x)/2+o(1)\] 
as $x \to \infty$. Thus at $n\lambda - \eps$ the LHS of \eqref{eqn:aequalsc} is  at least $\exp(n\lambda)/2+o(1)$ whereas the RHS is equal to $\exp(n\lambda-\eps)/2 + o(1)$. So the RHS is smaller that the LHS at $n\lambda - \eps$ if $n$ is large enough. Similarly, the RHS is larger than the LHS at $n\lambda+\eps$ if $n$ is large enough. This shows that $t_n$ is in the interval $(n\lambda - \eps, n\lambda + \eps)$ if $n$ is large enough. Since $\eps>0$ was arbitrary, $t_n = n \lambda + o(1)$.

Recall that $a(t_n)=c(t_n)$ and $\cosh(c/2)= \sinh^2(e/4)$. Since \[\arccosh(\sinh^2(x))=2x-\log 2+o(1)\] as $x \to \infty$ and $e(t_n)/4=n\lambda + o(1)$ we obtain
\[
a(t_n)=c(t_n)=2 \arccosh(\sinh^2(e(t_n)/4)) = 4n\lambda -2\log 2 +o(1)
\]
as $n \to \infty$.
\end{proof}

The next thing we need is an asymptotic lower bound for $\sigma(t_n)$.

\begin{lem}
We have $\sigma(t_n) \geq \left(1+\sqrt{2}\right)^{-n}$ if $n$ is large enough.
\end{lem}
\begin{proof}
According to \eqnref{eq:sigma} we have
\[\cosh(\sigma(t))=\coth^2(t) =  1 + \frac{1}{\sinh^2(t)}\]
so that
\[
\sinh^2(\sigma(t)/2) = \frac{\cosh(\sigma(t)) - 1}{2} =  \frac{1}{2\sinh^2(t)} \geq 2 \exp(-2t).
\]

Now $x \geq \sinh(x/2)$ for every $x \in [0,4.354]$. Morever $t_n \geq 1$ for every $n \geq 2$ (see the proof of \lemref{lem:a_shorter_than_f}), which implies that $\sigma(t_n) \leq \arccosh(\coth^2(1))\approx 1.141$. We conclude that
\[
\sigma(t_n) \geq \sinh(\sigma(t_n)/2) \geq \sqrt{2} \exp(-t_n)=\sqrt{2}\left(1+\sqrt{2}\right)^{-n +o(1)} \geq \left(1+\sqrt{2}\right)^{-n} 
\]
if $n$ is large enough, where we used \lemref{lem:estimate1} for the equality sign.
\end{proof}

\begin{remark}
Actually, $\sigma(t_n)$ is closer to $2\sqrt{2} \left(1+\sqrt{2}\right)^{-n}$, but the above is all we need.
\end{remark}

The previous two lemmata combined together yield the following estimate for the ratio of $a(t_n)$ over $\sigma(t_n)$. If the signed graph $\Gamma$ has girth larger than this, then the systoles of $X(\Gamma)$ are the $a$-, $b$- and $c$-curves according to \thmref{thm:systoles_closed_surface}.

\begin{cor} \label{cor:girth_bound}
There is a constant $K>0$ such that  $a(t_n) / \sigma(t_n) \leq K n \left(1+\sqrt{2}\right)^n$ for every $n\geq 2$.
\end{cor}

For small $n$, we can compute $t_n$, $a(t_n)$ and $\sigma(t_n)$ numerically to get a more explicit bound on the girth (see Table \ref{table}).

\begin{table}[htp] 
\centering
\begin{tabular}{|c|c|c|c|c|c|}
\hline
 $n$ & $t_n$ & $a(t_n)$ & $\sigma(t_n)$ & $a(t_n)/\sigma(t_n)$ \\ 
\hline 
2 & 1.745752 & 5.909039 &  0.503760  & 11.729861 \\
\hline
3 & 2.645975 &  9.256205 &  0.201312  & 45.979325 \\
\hline
4 & 3.526946 & 12.731803 &  0.083188 & 153.048057 \\
\hline
\end{tabular}
\caption{Approximate values of $t_n$ and $a(t_n)/\sigma(t_n)$ for small $n$} \label{table}
\end{table}

For $n=2$, the girth of $\Gamma$ needs to be at least $12$, hence the genus of $X(\Gamma)$ at least $13$. The minimal number of vertices needed for a $3$-regular graph $\Gamma$ to have girth $46$ is not known, but it is at least $3\cdot 2^{22}=12\,582\,912$. The corresponding surfaces $X(\Gamma)$ have genus at least $18\,874\,369$. The genus required for $n=4$ is at least $2\cdot3 \cdot 4^{76}$, which is more than the estimated number of stars in the observable universe \cite{Ghose}.

\subsection{Counting signed graphs} \label{subsec:ribbon_count}

In this subsection, we give a lower bound for the number of isomorphism classes of connected, $n$-regular, signed graphs with $(g-1)$ edges, girth larger than $a(t_n)/\sigma(t_n)$, and trivial automorphism group for $g$ sufficiently large. This proves \thmref{thmB} from the introduction, which we restate more precisely as follows.

\begin{thm} \label{Thm:Count}
Let $n\geq3$ and let $g$ be a positive integer such that $2(g-1)/n$ is also a positive integer. Let $N(n,g)$ be the number of local maxima $x$ of the systole function in $\M_g$ with $\sys(x)=L_n=a(t_n)$ whose automorphism group is trivial. Then, for $g$ large enough, 
we have
\[
N(n,g) \geq \alpha_n \big(\beta \, g \big)^{\left(1-\frac{2}{n} \right) g}, 
\] 
where $\beta>0$ is independent of $n$ and $g$, and $\alpha_n>0$ is independent of $g$ and satisfies
\[
\log \log \log \frac 1{\alpha_n} \sim_n n\log(1+\sqrt{2}). 
\]
\end{thm}

The asymptotic notation $f(x) \sim_x g(x)$ above means that
\[
\lim_{x \to \infty} \frac{f(x)}{g(x)} = 1. 
\]
We emphasize which variable is sent to infinity since the functions we compare may depend on other parameters.

We start by giving a lower bound for the number $S(n, E, w)$ of isomorphism classes of unlabelled, connected, $n$-regular, signed graphs with $E$ edges and girth at least $w$. To simplify matters, first consider labelled such graphs. 

Recall that a signed graph is a graph plus a cyclic ordering of the edges attached to every vertex and a choice of sign between consecutive edges such that the product of the signs around any vertex is negative. Recall also that the cyclic order around any vertex can be reversed (and two signs around each neighbor changed appropriately) without changing the isomorphism class of a signed graph. Thus a signed graph $\Gamma$ with labelled vertices has a total of $2^V$ isomorphic representations with the same vertex labels, where $V$ is the number of vertices in $\Gamma$.

Assume that a cyclic order has been chosen for the edges around each vertex (there are $(n-1)!$ cyclic orders on $n$ elements). Then there are $n$ signs to pick around each vertex, but since their product is required to be negative, any sign can be deduced from the remaining ones. Hence the number of admissible sign patterns around a vertex is $2^{n-1}$. The total number of sign patterns on the whole graph is therefore $2^{(n-1)V}$.

Now to count the number of isomorphism classes of unlabelled signed graphs, we have to take into account the fact that some graphs admit non-trivial automorphisms, which could result in overcounting. To remedy this, we restrict ourselves to underlying graphs that are asymmetric, that is, have trivial automorphism group.  Let $A(n, E, w)$ be the number of unlabelled, connected, asymmetric, $n$-regular graphs with $E$ edges, and of girth at least $w$. Then the above reasoning shows that
\begin{equation} \label{Eq:Ribbon} 
S(n, E, w) \geq 2^{(n-2)V} ((n-1)!)^V A(n, E, w). 
\end{equation}
Note that all the signed graphs with an asymmetric underlying graph are themselves 
asymmetric. 

To estimate $A(n, E, w)$ we combine a few results from graph theory. Let 
$U(E, n, w)$ be the number of unlabelled $n$-regular graphs with $E$ edges 
and girth at least $w$. In the literature, it is often assumed that the graphs are
simple, namely, that $w \geq 3$ (no monogons or bigons). To emphasize this
and to maintain a consistent notation, we use $U(E, n, 3)$ for the number
of unlabelled $n$-regular simple graphs with $E$ edges.

In \cite{Bollobas-Counting} Bollob\'as showed:

\begin{thm}[Bollob\'as] \label{Thm:Bol} 
For every $n \geq 3$, we have
\[
U(n,E,  3) \sim_E \exp\left( - \sum_{i=1}^{2} \frac{(n-1)^i}{2i} \right)
\cdot \frac{(2E)!}{2^E \, E! \, V! \, (n!)^V} 
\]
as $E \to \infty$, where $V=2E/n$ is the number of vertices in the graphs.
\end{thm} 

Wormald \cite{Wormald} strengthened this result to show:

\begin{thm}[Wormald] \label{Thm:Wor} 
For every $n\geq 3$ and every $\girth \geq 3$ we have
\[ 
U(n, E, \girth) \sim_E \exp\left( - \sum_{i=3}^{\girth - 1} \frac{(n-1)^i}{2i} \right) U(n, E, 3). 
\]
\end{thm}

Together, these two results imply that
\begin{equation} \label{Eq:Asymptotic}
U(n,E, \girth) \sim_E \exp\left( - \sum_{i=1}^{\girth -1} \frac{(n-1)^i}{2i} \right)
\cdot \frac{(2E)!}{2^E \, E! \, V! \, (n!)^V}. 
\end{equation} 

Bollob\'as also showed that regular simple graphs 
are generically connected \cite[p.195]{Bollobas-Connected} and asymmetric \cite[Theorem 6]{Bollobas-Asymmetric}.

\begin{thm}[Bollob\'as] 
\label{Thm:Asym} For every $n\geq 3$ we have
\[
A(n, E, 3) \sim_E U(n, E, 3). 
\]
\end{thm}

As a consequence, we have that a generic $n$-regular graph of girth at least $\girth\geq3$ is
also connected and asymmetric. Indeed,
\begin{align*}
\frac{U(n , E, w)- A(n, E, w)}{U(n, E, w)} 
&\sim_E 
 \exp\left(  \sum_{i=3}^{\girth - 1}  \frac{(n-1)^i}{2i} \right)
 \frac{U(n, E, w)- A(n, E, w)}{U(n, E, 3)}  \\ 
 &\leq
  \exp\left(  \sum_{i=3}^{\girth - 1}  \frac{(n-1)^i}{2i} \right)
  \frac{U(n, E, 3)- A(n, E, 3)}{U(n, E, 3)} \sim_E 0  
\end{align*}
where the first $\sim$ is \thmref{Thm:Wor}, the inequality holds because
graphs of girth at least $\girth$ form a subset of the set of 
graphs of girth at least $3$ and the last $\sim$ is \thmref{Thm:Asym}. 
This shows that
\[
A(n, E, w) \sim_E U(n , E, w)
\]
for every $n \geq 3$ and $w\geq 3$.

Combining this with \eqnref{Eq:Ribbon} and \eqnref{Eq:Asymptotic},
we get (after simplification)
\begin{equation} \label{Eq:Ribbon-Estimate}
S(n, E, \girth) \gtrsim_E 
\exp\left( - \sum_{i=1}^{\girth -1} \frac{(n-1)^i}{2i} \right)
\cdot \frac{(2E)!\, 2^E}{E! \, V! \, (4n)^V} . 
\end{equation} 

\begin{proof}[Proof of \thmref{Thm:Count}] 
By \thmref{thm:local_max} and \corref{cor:girth_bound}, for every finite connected $n$-regular signed graph $\Gamma$ of girth larger than
$K  n (1 + \sqrt{2})^n$, the surface $X(\Gamma)$ is a local maximum of the 
systole function at height $L_n=a(t_n)$ in $\M_g$ where $g=E+1$. Also, by \thmref{thm:auto}, non-isomorphic signed graphs correspond to distinct points in moduli space, and if the signed graph $\Gamma$ is asymmetric then $X(\Gamma)$ has trivial automorphism group too. In other words, the number of asymmetric local maxima of the systole function at height $L_n$  in $\M_g$ is at least $S(n,g-1,\lfloor K  n (1 + \sqrt{2})^n \rfloor + 1)$ .

Thus, we only need to simplify \eqnref{Eq:Ribbon-Estimate} and write it in terms 
of $n$ and $g$. Set
\[
\alpha_n = \exp\left( - \sum_{i=1}^{\girth -1} \frac{(n-1)^i}{2i} \right)
\]
where $w = \lfloor K  n (1 + \sqrt{2})^n \rfloor + 1$. Note that $\alpha_n$ depends only on 
$n$ and not on $g$. We have
\[
\frac{(n-1)^{w-1}}{2w}  \leq \log \frac 1{\alpha_n} \leq w n^w.
\]
Taking the logarithm two more times, we get
\begin{equation} \label{Eq:alpha_n}
\log \log \log \frac{1}{\alpha_n} \sim_n \log w \sim_n  n \log(1+\sqrt 2). 
\end{equation}
We use the Stirling's formula to simplify the remaining terms. The latter implies that there are positive constants 
$\beta_1, \dots, \beta_4$ such that
\[
(2E)!  \gtrsim_g (\beta_1 \, g)^{2g},  \quad
E! \lesssim_g (\beta_2 \, g)^{g},  \quad
V! \lesssim_g (\beta_3 \, g/n)^{2g/n}
\quad\text{and}\quad
(4n)^V \lesssim_g (\beta_4 \, n)^{2g/n}. 
\]
Hence, after collecting the constants, we can estimate the remaining terms in 
\eqnref{Eq:Ribbon-Estimate} as
\[
\frac{(2E)! \, 2^E}{E! \, V! \,(4n)^V}\gtrsim_g
\beta^{\, g} 
\frac{ g^{2g}}{ g^{g} \,(g/n)^{2g/n}\, n^{2g/n}}
\gtrsim_g  \left(\beta \, g\right)^{\left(1-\frac {2}{n} \right)g},
\]
for some constant $\beta>0$ independent of $n$ and $g$. This finishes the proof. 
\end{proof}

\end{document}